\documentclass[10pt]{article}
\usepackage{graphicx}
\usepackage{mathrsfs}
\usepackage{bbm}
\usepackage{amsfonts}
\usepackage{amssymb}
\usepackage{amsmath,amsthm}
\usepackage{float}

\usepackage{mathrsfs}
\usepackage{bbm}
\usepackage{amsfonts}
\usepackage{amssymb}
\usepackage{leftidx}
\usepackage{algorithm,algorithmic}
\usepackage[colorlinks,linktocpage,linkcolor=blue]{hyperref}
\usepackage[capitalise, nosort]{cleveref}
\usepackage{authblk}
\usepackage{url}

\usepackage{hyperref}
\usepackage{geometry}

\usepackage{booktabs,diagbox}

\def\opn#1#2{\def#1{\operatorname{#2}}} 
\opn\chara{char} \opn\length{\ell}
\allowdisplaybreaks
\opn\projdim{proj\,dim} \opn\injdim{inj\,dim}
\opn\rank{rank} \opn\depth{depth} \opn\grade{grade}
\opn\height{height} \opn\embdim{emb\,dim} \opn\codim{codim}

\opn\Tr{Tr} \opn\bigrank{big\,rank}
\opn\superheight{superheight}\opn\lcm{lcm}
\opn\trdeg{tr\,deg}%
\opn\reg{reg} \opn\lreg{lreg}
\opn\Ker{Ker} \opn\Coker{Coker} \opn\Im{Im} \opn\Hom{Hom}
\opn\Tor{Tor} \opn\Ext{Ext} \opn\End{End} \opn\Aut{Aut} \opn\id{id}

\opn\nat{nat}
\opn\pff{pf}
\opn\Pf{Pf} \opn\GL{GL} \opn\SL{SL} \opn\mod{mod} \opn\ord{ord}

\let\to=\rightarrow

\def\Implies{\ifmmode\Longrightarrow \else
	\unskip${}\Longrightarrow{}$\ignorespaces\fi}
\def\implies{\ifmmode\Rightarrow \else
	\unskip${}\Rightarrow{}$\ignorespaces\fi}
\def\iff{\ifmmode\Longleftrightarrow \else
	\unskip${}\Longleftrightarrow{}$\ignorespaces\fi}

\let\:=\colon

\newtheorem{Theorem}{Theorem}[section]

\newtheorem{Lemma}[Theorem]{Lemma}

\newtheorem{Proposition}[Theorem]{Proposition}
\newtheorem{Remark}[Theorem]{Remark}

\newtheorem{Definition}[Theorem]{Definition}

\theoremstyle{definition}

\let\epsilon=\varepsilon
\let\kappa=\varkappa
\opn\ini{in} \opn\inm{inm} \opn\Sym{Sym} \opn\diag{diag}
\opn\Ii{(i)} \opn\Iii{(ii)}

\usepackage[T1]{fontenc}
\usepackage[utf8]{inputenc}

\title{A new error analysis for parabolic Dirichlet boundary control problems\thanks{
Wei Gong was supported in part by the Strategic Priority Research Program of Chinese
Academy of Sciences (Grant No. XDB 41000000) and the National Natural Science Foundation of China
(Grant No. 12071468). Xiaoping Xie was supported in part by the National Natural Science Foundation of China  (Grant No. 12171340).}
}
\author[ ]{Dongdong Liang\thanks{The Hong Kong Polytechnic University Shenzhen Research Institute, Shenzhen 518057, China \& School of Mathematics, Sichuan University, Chengdu 610064, China. Email: dongdong.liang@polyu.edu.hk}
\quad Wei Gong\thanks{NCMIS \& LSEC, Institute of Computational Mathematics, Academy of Mathematics and Systems Science, Chinese Academy of Sciences, Beijing 100190, China. Email: wgong@lsec.cc.ac.cn }
\quad Xiaoping Xie\thanks{School of Mathematics, Sichuan University, Chengdu 610064, China. Email: xpxie@scu.edu.cn}}
\date{}
	
\begin{document}
	\maketitle
	
	{\bf Abstract:}\hspace*{5pt} { In this paper, we consider the finite element approximation to a parabolic Dirichlet boundary control problem and establish new a priori error estimates. In the temporal semi-discretization we apply  the DG(0) method for the state and the variational discretization for the control, and obtain the convergence rates $O(k^{\frac{1}{4}})$ and $O(k^{\frac{3}{4}-\varepsilon})$ $(\varepsilon>0)$  for the control  for problems posed on polytopes with $y_0\in L^2(\Omega),\,y_d\in L^2(I;L^2(\Omega))$  and smooth domains with $y_0\in H^{\frac{1}{2}}(\Omega),\,y_d\in L^2(I;H^1(\Omega))\cap H^{\frac{1}{2}}(I;L^2(\Omega)) $, respectively.  In  the fully discretization of the optimal control problem posed on polytopal domains, we apply the DG(0)-CG(1) method for  the state and the variational discretization approach for the control, and derive  the convergence order $O(k^{\frac{1}{4}} +h^{\frac{1}{2}})$, which improves the known results by removing the mesh size condition $k=O(h^2)$ between the space mesh size $h$ and the time step $k$. 
As a byproduct, we obtain a priori error estimate $O(h+k^{1\over 2})$ for the fully discretization of parabolic equations with inhomogeneous Dirichlet data posed on polytopes, which also improves the known error estimate by removing the above mesh size condition. }
	
	{{\bf Keywords:}\hspace*{10pt} 
	parabolic equation, Dirichlet boundary control, finite element method, semi-discretization, full  discretization, error estimate}
	
	
	\section{Introduction}
	\setcounter{equation}{0}
	Let $\Omega\subseteq \mathbb{R}^N (N=2,3)$ be a bounded, convex polytopal or smooth domain with boundary $\Gamma:=\partial\Omega$.  Let $T>0$ be a constant, and denote   $I:=(0,T)$.  In this article, we study the following parabolic Dirichlet boundary optimal control problem:
	\begin{equation}\label{BC_Functioal}
		\min\limits_{u\in U_{ad},\ y\in L^2(I;L^2(\Omega))}\quad J(y,u)=\frac{1}{2}\Vert y-y_d\Vert^2_{L^2(I;L^2(\Omega))}+\frac{\alpha}{2}\Vert u\Vert^2_{L^2(I;L^2(\Gamma))},
	\end{equation}
where $\alpha>0$ denotes the regularization parameter, $y_d\in L^2(I;L^2(\Omega))$ is a given target state, $U_{ad}$ is the admissible set of controls of the following box type:
	\begin{equation}
		U_{ad}=\big\{u\in L^2(I;L^2(\Gamma)):\quad u_a\le u(t,x)\le u_b,\quad  \mbox{a.e.}\  (t,x)\in \Sigma_T\big\},\nonumber
	\end{equation}
	with $u_a,\,u_b\in \mathbb{R}\cup\{\infty\}$ satisfying  $u_a<u_b$. The state variable $y$ and the control variable $u$ in the optimal control problem (\ref{BC_Functioal}) are constrained by the following parabolic equation:
	\begin{equation}\label{BC_stateE}
		\begin{cases}
			\partial_t y-\Delta y=f &\mbox{in}\  \Omega_T:=I\times\Omega, \\ 
			y=u   &\mbox{on}\ \Sigma_T:=I\times\Gamma ,\\
			y(\cdot,0)=y_0 &\mbox{in}\ \Omega,
		\end{cases}
	\end{equation}
	where $f\in L^2(I;L^2(\Omega))$ and $y_0\in L^2(\Omega)$ are given. The solution pair $(\bar{u},\bar{y})$ with $\bar{u}\in U_{ad}$ minimizing the cost functional $J$ is called the optimal pair of the optimal control problem \eqref{BC_Functioal}.
	
Dirichlet boundary control is a class of important control problems and plays a key role in many practical applications, especially in fluid dynamics, e.g., the active boundary control for fluid flows \cite{A.V. Fursikov,MDGunzburger1991,MHinzeandKKunisch}, where the angular velocity along the axes  of a cylinder is taken as the control to adjust the velocity of fluid flowing the surface of it. The target of such kind of  controls is to restrain the falling off of vortex or to postpone the production of turbulence. In general, controls with low regularity are admissible, such as blowing or suction on a part of  the boundary as controls, which could have jumps and satisfy  pointwise constraints.
	
In the past few decades, Dirichlet boundary control problems have been extensively studied; we refer to \cite{Thomas Apel, ECasasandJPRaymond, ECasasandJSokolowski, DeckelnickGuntherHinze2009, DAFrenchandJTKing} for elliptic Dirichlet boundary controls,   to \cite{A.V. Fursikov, MHinzeandKKunisch} for   boundary controls  of Navier-Stokes equations, and to \cite{NAradaandJPRaymond,KKunischandBVexler2007,ILasiecka1980,Lasiecka,J. L. LIONS}  for  parabolic Dirichlet boundary control problems. The study of Dirichlet  boundary control problems differs in the choice of the control space. The $L^2(\Gamma)$-control space (cf. \cite{Thomas Apel, ECasasandJPRaymond, DeckelnickGuntherHinze2009, DAFrenchandJTKing}) is very popular and easy to implement but usually yields  less regular solutions, while the energy space method (cf. \cite{Of}) gives smoother solutions but the implementation is more involved. As far as  the parabolic Dirichlet boundary control problem \eqref{BC_Functioal} is concerned, when it
is formulated with the control space $L^2(I;L^2(\Gamma))$, the state equation has to be understood in the very weak sense (cf. \cite{M. Berggren-2004860}) and the optimal control often behaves less regular  \cite{NAradaandJPRaymond,KKunischandBVexler2007,ILasiecka1980,Lasiecka,J. L. LIONS} , e.g., admits only $H^{\frac{1}{2}}$-regularity in space and $H^{\frac{1}{4}}$-regularity in time in polytopes. To alleviate this difficulty, a Robin penalization approach was proposed in  \cite{CBernardiandHEFekih}.

When comes to discretizations and error estimates for the Dirichlet boundary control problem, the finite element method has been extensively used  and there are many contributions for elliptic control problems, see e.g., \cite{Thomas Apel,ECasasandJPRaymond,ECasasandJSokolowski,DeckelnickGuntherHinze2009,DAFrenchandJTKing,SMayRRannacheran}. However, there are only a few works for parabolic Dirichlet boundary control problems, cf. \cite{WM-Hinze,WGongandBYLi,KKunischandBVexler2007}. In \cite{KKunischandBVexler2007},   a fully discretization of the optimal control problem was considered by using the  DG(r)-CG(s) method  and a  semi-smooth Newton method was applied to solve the optimization problem. In \cite{WM-Hinze}, the   DG(0)-CG(1) method combining with the variational discretization was used to obtain the discrete optimal control problem in two dimensions, and   the convergence orders $O(h^\frac{1}{2})$ in space and $O(k^\frac{1}{4})$ in time were obtained under the condition $k=O(h^2)$, where $k$ and $h$ are the time step and mesh size, respectively. In \cite{WGongandBYLi}, an improved convergence order $O(h^{1-\frac{1}{q}-r})$ was derived for a spatial semi-discretization, where $q>0$ depends on the maximal interior angle of the corners and $r>0$ is an arbitrary small constant. In this paper, we intend to derive new a priori error estimates without the mesh condition $k=O(h^2)$ by separating the error estimates for the temporal and spatial discretizations. Our results of temporal semi-discretizations apply to problems posed on both convex polytopal and smooth domains, and the results of fully discretizations are valid only for polytopal domains.
	
To carry out the error analysis of  finite element discretization  of  the Dirichlet boundary  control problem \eqref{BC_Functioal}, an important and indispensable ingredient is  the error estimation and stability of the finite element approximation of parabolic equations with rough Dirichlet boundary conditions, see e.g., \cite{DAFrenchandJTKing1993,WM-Hinze,ILasiecka1986}. In \cite{ILasiecka1986},  the convergence orders $O(h^{\frac{1}{2}})$ $(1<p<\infty)$ and $O(h^{\frac{1}{2}}\big|\ln {\,h}\big|)$ were obtained for the  spatial semi-discretization under the norm $\Vert\cdot\Vert_{L^p(I;L^2(\Omega))}$ for $p=1$  and $\infty$, respectively, by assuming  the boundary data to be in space  $L^p(I;L^2(\Gamma))$. For the parabolic equation with Dirichlet data in $L^2(I;L^2(\Gamma))$,   the convergence orders $O(k^\frac{1}{4})$ in time and $O(h^\frac{1}{2})$ in space  were obtained in \cite{DAFrenchandJTKing1993,WM-Hinze}  for the fully discretization scheme under the condition $k=O(h^2)$ for the two dimensional case. We remark that the backward Euler or DG(0) scheme used in \cite{DAFrenchandJTKing1993,WM-Hinze} is unconditionally stable so that the condition $k=O(h^2)$, caused by the inverse estimate with the low regularity of solutions, is restrictive  there. In our analysis, we will  remove this restriction. 
	
In  this paper, we consider the temporal semi-discretization and time-space full discretization of the parabolic Dirichlet boundary control problem \eqref{BC_Functioal}, and derive the corresponding error estimates. 
We apply the DG(0) method to the state and the variational discretization to the control for the temporal semi-discretization. If $f\in L^2(I;L^2(\Omega)),\ y_0\in H^{2s-1}(\Omega)$ and $y_d\in L^2(I;H^{2s^\prime-\frac{1}{2}}(\Omega))\cap H^{s^\prime-\frac{1}{4}}(I;L^2(\Omega))$ with $s\in[\frac{1}{2}, \frac{3}{4})$, $s^\prime\in [s-\frac{1}{4}, s]$ in case that $\Omega$ is smooth and $s=\frac{1}{2},\ s^\prime=\frac{1}{4}$ in case that $\Omega$ is a polytope,  we can obtain the following error estimate:

		\begin{equation}
			\Vert \bar{u}-\bar{u}_k\Vert_{L^2(I;L^2(\Gamma))}+\Vert \bar{y}-\bar{y}_k\Vert_{L^2(I;L^2(\Omega))}\le C(k^s+k^{s^\prime}),
		\end{equation}
where $(\bar{u},\bar{y})$ and $(\bar{u}_k,\bar{y}_k)\in U_{ad}\times X_k$ are the optimal solution of the optimal control problem \eqref{BC_Functioal} and the semi-discrete optimal control problem \eqref{time_semi_control}, respectively, 
under the non-increasing condition on time steps. For the full discretization where $\Omega$ is a polytope, we use the DG(0)-CG(1) method to the state and variational discretization to the control, and obtain the convergence order $O(h^\frac{1}{2})$  for the error between the semi-discrete and fully discrete solutions. Finally, combining the above two error estimates, we   obtain the following error estimate:
    \begin{equation}
    	\Vert\bar{u}-\bar{u}_{kh}\Vert_{L^2(I;L^2(\Gamma))}+\Vert \bar{y}-\bar{y}_{kh}\Vert_{L^2(I;L^2(\Omega))}\le C(h^{\frac{1}{2}}+k^{\frac{1}{4}}),
    \end{equation}
    where $(\bar{u},\bar{y})$ and $(\bar{u}_{kh},\bar{y}_{kh})$ denote the continuous  and fully discrete optimal pairs, respectively. As a byproduct, we obtain the following a priori error estimate for the fully discretization of parabolic equations with inhomogeneous Dirichlet data posed on polytopes
     \begin{equation}
    	\Vert {y}(u)-{y}_{kh}(u)\Vert_{L^2(I;L^2(\Omega))}\le C(h+k^{\frac{1}{2}}),
    \end{equation}
 where $y(u)$ and ${y}_{kh}(u)$ denote the solution of (\ref{BC_stateE}) and its fully discretization (\ref{spatial_dicrete_state}) for any given $u\in L^2(I;H^{1\over 2}(\Gamma))\cap H^{1\over 4}(I;L^2(\Gamma))$, respectively. To derive the above error estimates we develop a new error analysis by separating the error estimates for the temporal and spatial discretizations, which is more involved and technical compared to \cite{MeidnerVexler,MeidnerVexler2} due to the inhomogeneous Dirichlet data and low regularity solutions.

 The rest of this paper is organized as follows: In section 2 we present some preliminaries including the definition of very weak solutions and the well-posedness of parabolic equations. In section 3, we derive the first order optimality condition and the regularity of the optimal pair for the parabolic Dirichlet boundary control problem. In section 4, we present the DG(0) semi-discretization and DG(0)-CG(1) full discretization of parabolic boundary control problems, and give  the corresponding discrete first order optimality conditions. We devote section 5 to  the error analysis of   the temporal discretization and spatial discretization of inhomogeneous parabolic boundary value problems. Finally,   in section 6 we derive the error estimates for the  semi-discrete and   fully discrete optimal control problems.

	\section{Preliminaries}
	\setcounter{equation}{0}

Throughout the paper we follow the standard notations for differential operators, function spaces and norms
that can be found, for example, in \cite{BrennerScott2008,J.L1,V.Thom}, and assume that $C>0$ is a constant independent of $h$, $k$ and the given data.
	
Recall that in the optimal control problem (\ref{BC_Functioal}) we assume that
\begin{equation}\label{data-cond}
 f\in L^2(I;L^2(\Omega)), \quad u\in L^2(I;L^2(\Gamma)),  \quad y_0\in L^2(\Omega),
\end{equation}
 the solution of the state equation (\ref{BC_stateE}) has to be defined by using the method of transposition \cite{M. Berggren-2004860,J.L1}. This kind of solution  is called very weak solution whose definition is given below.

\begin{Definition}\label{def_T}
	For any given  $f $, $u $ and $y_0 $ satisfying \eqref{data-cond}, a function $y\in L^2(I;L^2(\Omega))$ is called the very weak solution of \eqref{BC_stateE}, if $y$ satisfies the identity
	\begin{equation}\label{IE_1}
		\int_{\Omega_T}y(-\partial_tz-\Delta z)dxdt=-\int_{\Sigma_T}u\partial_nzdsdt+\int_{\Omega_T}fzdxdt+\int_{\Omega}y_0z(\cdot,0)dx
	\end{equation}
	for any $z\in L^2(I;H^2(\Omega)\cap H^1_0(\Omega))\cap H^1(I;L^2(\Omega))$ with $z(\cdot,T)=0$. Equivalently, there holds
	\begin{equation}\label{IE_2}
		\int_{\Omega_T}ygdxdt=-\int_{\Sigma_T}u\partial_nzdsdt+\int_{\Omega_T}fzdxdt+\int_{\Omega}y_0z(\cdot,0)dx\quad\forall g\in L^2(I;L^2(\Omega)),
	\end{equation}
	 where $\partial_nz$ denotes the outward normal derivative of $z$, and $z$ satisfies
	\begin{equation}\label{T_back_eq}
		\begin{cases}
			-\partial_tz-\Delta z=g\quad&\mbox{in}\ (0,T)\times\Omega,\\
			z=0 \qquad &\mbox{on}\ (0,T)\times\Gamma,\,\\
			z(\cdot,T)=0\qquad&\mbox{in}\ \Omega.
		\end{cases}
	\end{equation}
\end{Definition}
\begin{Remark}
	By the well-posedness of parabolic equations (cf. \cite[Chapter 7]{L. C. EVANS}), we know that $z\in L^2(I;H^2(\Omega)\cap H^1_0(\Omega))\cap H^1(I;L^2(\Omega))$, therefore $\partial_n z\in L^2(I;H^{\frac{1}{2}}(\Gamma))$ by the trace theorem (cf. \cite[Chapter 4]{J.L1} and \cite{ECasasandJPRaymond}). Using the embedding relation
	\begin{equation}
		L^2(I;H^2(\Omega)\cap H^1_0(\Omega))\cap H^1(I;L^2(\Omega))\hookrightarrow C(\bar{I};H^1(\Omega)),\nonumber
	\end{equation}
	we have $z\in C(\bar{I};L^2(\Omega))$. Hence, Definition \ref{def_T} is well-defined.
\end{Remark}

The existence and uniqueness of a very weak solution of equation \eqref{BC_stateE} is verified in the following lemma; see also \cite{WM-Hinze,GongHinzeZhou}.
\begin{Lemma}\label{T_Solution_E}
	For any given $y_0 $, $u $ and $f $ satisfying \eqref{data-cond}, there exists a unique very weak solution $y\in L^2(I;L^2(\Omega))$ of equation \eqref{BC_stateE}. Moreover, there  holds
	\begin{equation}
		\Vert y\Vert_{L^2(I;L^2(\Omega))}\le C \big(\Vert y_0\Vert_{L^2(\Omega)}+\Vert u\Vert_{L^2(I;L^2(\Gamma))}+\Vert f\Vert_{L^2(I;L^2(\Omega))}\big).\nonumber
	\end{equation}
\end{Lemma}

In addition, if the data of equation (\ref{BC_stateE}) have higher regularity, we can expect improved regularity of the solution which depends also on the smoothness of $\Omega$.
\begin{Lemma}\label{state_equation_regularity}
	Suppose that $\Omega$ is a bounded smooth domain. For any given $f\in L^2(I;L^2(\Omega))$,    $$u\in H^{s-\frac{1}{4}}(I;L^2(\Gamma))\cap L^2(I;H^{2s-\frac{1}{2}}(\Gamma))\  \text{ and } \ y_0\in H^{2s-1}(\Omega)$$
	with $\frac{1}{4}\le s<\frac{3}{4}$,    the solution $y$ of equation \eqref{BC_stateE} belongs to $ L^2(I;H^{2s}(\Omega))\cap H^s(I;L^2(\Omega))$ and satisfies:
	\begin{align*}
		\Vert y\Vert_{L^2(I;H^{2s}(\Omega))}+\Vert y\Vert_{H^s(I;L^2(\Omega))}&\le C\big(\Vert f\Vert_{L^2(I;L^2(\Omega))}+\Vert y_0\Vert_{H^{2s-1}(\Omega)}+\Vert u\Vert_{H^{s-\frac{1}{4}}(I;L^2(\Gamma))}\\
		&\qquad +\Vert u\Vert_{L^2(I;H^{2s-\frac{1}{2}}(\Gamma))}\big).\nonumber
	\end{align*}
\end{Lemma}

The regularity of solution to (\ref{T_back_eq}) can also be improved if we have sufficient smoothness on the data and the domain $\Omega$. When $\Omega$ is a convex polytope, the weak solution of  equation (\ref{T_back_eq}) belongs to $H^1(I;L^2(\Omega))\cap L^2(I;H^1_0(\Omega)\cap H^2(\Omega)) $ for any $g\in L^2(I;L^2(\Omega))$. In case that $\Omega$ has a smooth boundary, the regularity of  solution  to  (\ref{T_back_eq}) is presented in the following lemma (\cite{L. C. EVANS,J.L1}).
\begin{Lemma}\label{back_regularity}
	Let $\Omega$ be a smooth bounded domain. For any given $0\le r<\frac{1}{2}$ and function $g\in L^2(I;H^{2r}(\Omega))\cap H^r(I;L^2(\Omega))$, the solution $z$ of equation \eqref{T_back_eq} belongs to $ L^2(I;H^{2+2r}(\Omega))\cap H^{1+r}(I;L^2(\Omega))$ and there holds the following estimate:
	\begin{equation}
		\Vert z\Vert_{L^2(I;H^{2+2r}(\Omega))}+\Vert z\Vert_{H^{1+r}(I;L^2(\Omega))}\le C\big(\Vert g\Vert_{L^2(I;H^{2r}(\Omega))}+\Vert g\Vert_{H^r(I;L^2(\Omega))}\big).\nonumber
	\end{equation}
If, in addition, $g\in L^2(I;H^1(\Omega))\cap H^{\frac{1}{2}}(I;L^2(\Omega))$ and
\begin{equation}\label{compatibility_condition}
-\partial_t z(T)-\Delta z(T)=g(T)\quad \mbox{in}\ \Omega,
\end{equation}
then $ z\in L^2(I;H^3(\Omega))\cap H^{\frac{3}{2}}(I;L^2(\Omega))$ (cf. \cite[Page 32]{J.L1}) and $\partial_t z\in H^{\frac{1}{2}}(I;L^2(\Omega))\cap L^2(I;H^1(\Omega))$ (cf. \cite[Page 12]{J.L1}). Moreover, there holds the  estimate
\begin{equation}
\Vert z\Vert_{L^2(I;H^3(\Omega))}+\Vert z\Vert_{H^{\frac{3}{2}}(I;L^2(\Omega))}+\left\Vert \partial_t z\right\Vert_{L^2(I;H^1(\Omega))}\le C(\Vert g\Vert_{L^2(I;H^1(\Omega))}+\Vert g\Vert_{H^{\frac{1}{2}}(I;L^2(\Omega))}).
\end{equation}
\end{Lemma}

The above regularity results   will be applied to study the regularity of solutions to the optimal control problem (\ref{BC_Functioal}). 
In general, Lemma \ref{state_equation_regularity} will be used to obtain the regularity of state variables, and Lemma \ref{back_regularity} is helpful to improve  the regularity of control variables.

\section{The optimal control problem}
\setcounter{equation}{0}
For any $u\in L^2(I;L^2(\Gamma))$, the unique solution of the state equation (\ref{BC_stateE}), defined in Definition \ref{def_T}, is denoted by $y(u)$. Therefore, we can define the following linear operator:
\begin{equation}
	S:L^2(I;L^2(\Gamma))\to L^2(I;L^2(\Omega)),\ \quad \ \ Su:=y(u),
\end{equation}
which is bounded and Fr\'{e}chet differentiable from $L^2(I;L^2(\Gamma))$ to $L^2(I;L^2(\Omega))$. Then the optimal control problem \eqref{BC_Functioal} can be equivalently written as
\begin{equation}\label{optimal_P}
	\min\limits_{u\in U_{ad}}\hat{J}(u):=J(Su,u).
\end{equation}
It is easy to check that the above optimization problem admits a unique solution $\bar u$ (cf. \cite[Chapter \uppercase  \expandafter{\romannumeral3}]{J. L. LIONS} or \cite[Chapter  1]{HinzePinnauUlbrich}). We denote by $(\bar{u},\bar{y})$ the optimal pair with $\bar y$ the associated state. 

For any $u\in L^2(I;L^2(\Gamma))$, the Frech\'et derivative of the cost functional at $u$ reads
\begin{equation}
	\hat{J}^{'}(u)v=\int_{\Sigma_T}\alpha uvdsdt+\int_{\Omega_T}(y-y_d)\tilde{y}(v)dxdt\qquad\forall v\in L^2(I;L^2(\Gamma)),\nonumber
\end{equation}
where $\tilde{y}(v)\in L^2(I;L^2(\Omega))$ is the solution of the following equation:
\begin{equation}
	\begin{cases}
		\partial_t \tilde{y}(v)-\Delta\tilde{y}(v)=0\quad&\mbox{in}\ (0,T)\times\Omega,\\
		\tilde{y}(v)=v \quad\,&\mbox{on}\ (0,T)\times\Gamma, \\\nonumber
		\tilde{y}(v)(\cdot,0)=0\quad\,&\mbox{in}\ \Omega.
	\end{cases}
\end{equation}

The first order sufficient and necessary optimality condition is given in the following theorem.
\begin{Theorem}\label{Necess_opti_condi}
	The pair $(\bar{u},\bar y)\in U_{ad}\times L^2(I;L^2(\Omega))$ is the optimal solution of the optimal control problem \eqref{BC_Functioal} if and only if $\bar y:=S\bar u=y(\bar u)$ and
	\begin{equation}\label{V_Inequation1}
		\hat{J}^{'}(\bar{u})(v-\bar{u})\ge0\qquad\forall v\in U_{ad}.
	\end{equation}
	Furthermore, \eqref{V_Inequation1} is equivalent to
	\begin{equation}\label{V_Inequation2}
		\hat{J}^{'}(\bar{u})(v-\bar{u})=\int_{\Sigma_T}(\alpha \bar{u}-\partial_n \bar{z})(v-\bar{u})dsdt\ge0\qquad\forall v\in U_{ad},
	\end{equation}
	where $\bar{z}\in L^2(I;H^2(\Omega)\cap H^1_0(\Omega))\cap H^1(I;L^2(\Omega))$ is the so-called adjoint variable satisfying
	\begin{equation}\label{Conjugate_EQ}
		\begin{cases}
			-\partial_t \bar{z}-\Delta \bar{z}=\bar{y}-y_d\quad\,&\mbox{in}\ (0,T)\times\Omega,\\
			\bar{z}=0 \qquad\,\quad\,\,& \mbox{on}\ (0,T)\times\Gamma, \\
			\bar{z}(\cdot,T)=0\qquad\,\quad\,\,&\mbox{in}\ \Omega\,.
		\end{cases}
	\end{equation}
\end{Theorem}


	The optimality condition (\ref{V_Inequation2}) can be equivalently written as
	\begin{equation}\label{optimal_control}
		\bar{u}=P_{U_{ad}}\Big(\frac{1}{\alpha}\partial_n \bar{z}\Big),
	\end{equation}
	where $P_{U_{ad}}$ is the orthogonal projection operator from $L^2(I;L^2(\Gamma))$ onto the admissible set $U_{ad}$ that preserves the $H^s$-regularity $(s\le 1)$ (cf. \cite[Page 114]{F.Trltzsch-2004860}). The above relationship between the optimal control and adjoint state allows us to improve the regularity of solutions to the optimal control problem. 
\begin{Theorem}\label{optimall_regularity}
Assume that $\Omega$ is a bounded smooth domain and let $(\bar{u}, \bar{y}, \bar{z})$
 be the optimal solution of the optimal control problem \eqref{BC_Functioal}. Then for any given  $y_0\in H^{2s-1}(\Omega)\ (\frac{1}{2}\le s<\frac{3}{4})$, $y_d\in L^2(I;H^{2s^\prime-\frac{1}{2}}(\Omega))\cap H^{s^\prime-\frac{1}{4}}(I;L^2(\Omega))\ (s-\frac{1}{4}\le s^\prime\le s)$ and $f\in L^2(I;L^2(\Omega))$, there hold
		\begin{align*}
			\bar{y}\in L^2(I;H^{2s}(\Omega))\cap H^s(I;L^2(\Omega)),\quad \bar{u}\in L^2(I;H^{{\rm min}\{2s^\prime,1\}}(\Gamma))\cap H^{s^\prime}(I;L^2(\Gamma)),\\
			\bar{z}\in L^2(I;H^{2s^\prime+\frac{3}{2}}(\Omega)\cap H^1_0(\Omega))\cap H^{s^\prime+\frac{3}{4}}(I;L^2(\Omega)),\\
 \partial_t\bar{z}\in L^2(I;H^{2(s^\prime-\frac{1}{4})}(\Omega))\cap H^{s^\prime-\frac{1}{4}}(I;L^2(\Omega)).
		\end{align*}
If, in addition, there holds the compatibility condition \eqref{compatibility_condition} with $g:=\bar{y}-y_d$ in the case $s^\prime=\frac{3}{4}$, then the adjoint state has the following improved regularity:
$$\bar{z}\in L^2(I;H^3(\Omega)\cap H^1_0(\Omega))\cap H^{\frac{3}{2}}(I;L^2(\Omega)),\quad \partial_t \bar{z}\in H^{\frac{1}{2}}(I;L^2(\Omega))\cap L^2(I;H^1(\Omega)).$$
	\end{Theorem}
	\begin{proof}
		For  $\bar{u}\in L^2(I;L^2(\Gamma))$, $f\in L^2(I;L^2(\Omega))$ and $y_0\in L^2(\Omega)$, it follows from Lemma \ref{T_Solution_E} that  $\bar{y}\in L^2(I;L^2(\Omega))$. Choosing $r=0$ in Lemma \ref{back_regularity} and combining with $y_d\in L^2(I;L^2(\Omega))$, we have $\bar{z}\in L^2(I;H^2(\Omega)\cap H^1_0(\Omega))\cap H^1(I;L^2(\Omega))$. Then $\bar{u}\in L^2(I;H^{\frac{1}{2}}(\Gamma))\cap H^{\frac{1}{4}}(I;L^2(\Gamma))$ by applying the identity (\ref{optimal_control}) and the trace theorem (cf. \cite{KKunischandBVexler2007}).

	Next, choosing $s=\frac{1}{2}$ in Lemma \ref{state_equation_regularity} and using the fact $y_0\in L^2(\Omega)$, $f\in L^2(I;L^2(\Omega))$, we obtain $\bar{y}\in L^2(I;H^1(\Omega))\cap H^{\frac{1}{2}}(I;L^2(\Omega))$. That is, $\bar{y}\in L^2(I;H^{2s-1}(\Omega))\cap H^{s-\frac{1}{2}}(I;L^2(\Omega))$ for $\frac{1}{2}\le s<\frac{3}{4}$. Then, it follows $\bar{z}\in L^2(I;H^{2s+1}(\Omega)\cap H^1_0(\Omega))\cap H^{s+\frac{1}{2}}(I;L^2(\Omega))$ from Lemma \ref{back_regularity} and the fact $\bar{y}-y_d\in L^2(I;H^{2s-1}$ $(\Omega))\cap H^{s-\frac{1}{2}}(I;L^2(\Omega))$ for $y_d\in L^2(I;H^{2s^\prime-\frac{1}{2}}(\Omega))\cap H^{s^\prime-\frac{1}{4}}(I;L^2(\Omega))\,\left(s-\frac{1}{4}\le s^\prime\right)$, which implies that $\bar{u}\in L^2(I;H^{2s-\frac{1}{2}}(\Gamma))\cap H^{s-\frac{1}{4}}(I;L^2(\Gamma))$ by the identity (\ref{optimal_control}) and the trace theorem (cf. \cite{KKunischandBVexler2007}). Further, observing that $y_0\in H^{2s-1}(\Omega)$ and $ f\in L^2(I;L^2(\Omega))$, then applying Lemma \ref{state_equation_regularity} to the state equation (\ref{BC_stateE}) implies that  $\bar{y}\in L^2(I;H^{2s}(\Omega))\cap H^s(I;L^2(\Omega))$. Then, one derives $\bar{z}\in L^2(I;H^{2s^\prime+\frac{3}{2}}(\Omega)\cap H^1_0(\Omega))\cap H^{s^\prime+\frac{3}{4}}(I;L^2(\Omega))$ from Lemma \ref{back_regularity} and the fact $\bar{y}-y_d\in L^2(I;H^{2s^\prime-\frac{1}{2}}(\Omega))\cap H^{s^\prime-\frac{1}{4}}(I;L^2(\Omega))$ for $y_d\in L^2(I;H^{2s^\prime-\frac{1}{2}}(\Omega))\cap H^{s^\prime-\frac{1}{4}}(I;L^2(\Omega))$ $(s^\prime\le s)$, which implies that $\bar{u}\in L^2(I;H^{{\rm min}(2s^\prime,1)}(\Gamma))\cap H^{s^\prime}(I;L^2(\Gamma))$ by \eqref{optimal_control}. At last, $\partial_t\bar{z}\in L^2(I;H^{2(s^\prime-\frac{1}{4})}(\Omega))\cap H^{s^\prime-\frac{1}{4}}(I;L^2(\Omega))$ can be derived from \cite[Page 12]{J.L1}. This completes the proof.
	\end{proof}

	
The regularity of solutions for the state and adjoint equations depends heavily on the smoothness of domains. From the above theorem, we see that the solution of optimal control problems in the smooth domain possesses higher regularity if the data is smooth, but this is not the case for the polytopal domain. The following theorem from \cite{WM-Hinze,WGongandBYLi} shows the regularity of the optimal control problem in polytopes.
	\begin{Theorem}\label{regularity_cont_cnvex}
		Let $\Omega$ be a convex polytope and $(\bar{u},\bar{y},\bar{z})$ be the solution of the optimal control problem \eqref{BC_Functioal}. Then for any given $y_d,\,f\in L^2(I;L^2(\Omega))$ and $y_0\in L^2(\Omega)$ there hold
		\begin{align*}
			\bar{y}\in L^2(I;H^1(\Omega))\cap H^{\frac{1}{2}}(I;L^2(\Omega)),\quad
			\bar{u}\in L^2(I;H^{\frac{1}{2}}(\Gamma))\cap H^{\frac{1}{4}}(I;L^2(\Gamma)),\\
			\bar{z}\in L^2(I;H^2(\Omega)\cap H^1_0(\Omega))\cap H^1(I;L^2(\Omega)).
		\end{align*}
	\end{Theorem}
	\begin{Remark}
		The results of Theorems \ref{optimall_regularity} and \ref{regularity_cont_cnvex} on the regularity of the optimal control problem \eqref{BC_Functioal} play an important role in the error estimation of  discrete optimal control problems. Theorem \ref{regularity_cont_cnvex} implies that the discrete control can achieve optimal  convergence order by using the DG(0)-CG(1) discretization scheme. 
	\end{Remark}

	\section{Finite element discretization}
	\setcounter{equation}{0}
	\subsection{Notations for finite element methods}
	In order to introduce the DG(0) semi-discretization in time and the DG(0)-CG(1) fully discretization in time-space of the state and adjoint equations, we first give some notations for finite element methods.

	    To begin with, we divide the interval $\bar{I}:=[0,T]$ into a family of subintervals ${I_m}:=(t_{m-1},t_m]$ with step sizes $k_m:=t_m-t_{m-1},\,m=1,\cdots,M,$ where $0=t_0<t_1<\cdots<t_M=T$. The maximal time step size is denoted by $k:=\max \limits_{1\le m\le M}k_m$.

    Throughout this article, we assume the following two conditions for $k_m$ and $k$:
    \begin{itemize}
    	\setlength{\itemsep}{1.5pt}
    	\setlength{\parsep}{1.2pt}
    	\setlength{\parskip}{1.2pt}
    	\item[(i)] There exists a constant $C>0$ independent of $k$ and $m$ such that
    	\begin{equation}\label{condition1}
    		\setlength{\abovedisplayskip}{8pt}
    		\setlength{\belowdisplayskip}{8pt}
    		k/k_m\le C,\quad m=1,\cdots,M.
    	\end{equation}
    	\item[(ii)] The time step  $k_m$ is non-increasing for $m$, i.e.,
    	\begin{equation}\label{condition2}
    		\setlength{\abovedisplayskip}{6pt}
    		\setlength{\belowdisplayskip}{6pt}
    		k_m/k_{m-1}\le 1,\quad m=2,\cdots,M.
    	\end{equation}
    \end{itemize}
The first condition ensures that the partition for $I$ is quasi-uniform, and the second one implies that the time step is non-increasing along the direction $T$.

    We define the temporal semi-discrete DG(0) space consisting of piecewise constants in time as follows:
    \begin{equation}
	\begin{split}
		X^0_k:=&\big\{v_k\in L^2(I;H^1_0(\Omega)):\ v_k|_{I_m}\in P_0(I_m;H^1_0(\Omega)),\ m=1,\cdots,M\big\},\\ \nonumber
		X_k:=&\big\{v_k\in L^2(I;H^1(\Omega)):\, v_k|_{I_m}\in P_0(I_m;H^1(\Omega)),\ m=1,\cdots,M\big\},\\
		\tilde{X}_k:=&\big\{v_k\in L^2(I;L^2(\Omega)):\, v_k|_{I_m}\in P_0(I_m;L^2(\Omega)),\ m=1,\cdots,M\big\},
	\end{split}
   \end{equation}
   where $P_0(I_m;H^1_0(\Omega))$, $P_0(I_m;H^1(\Omega))$ and $P_0(I_m;L^2(\Omega))$ denote function spaces of constants defined in $I_m$ $(m=1,\cdots,M)$ and valued in spaces $H^1_0(\Omega), H^1(\Omega)$ and $L^2(\Omega)$, respectively.

   {In addition, we   introduce the following two temporal semi-discrete finite element spaces consisting of  functions of piecewise constants in time:
	\begin{equation}
		\begin{split}
			X_k(\Gamma):&=\big\{v_k\in L^2(I;H^{\frac{1}{2}}(\Gamma)):\,\exists\,\tilde{v}_k\in X_k,\,\mbox{such that}\,v_k|_{I_m}=\tilde{v}_k|_{I_m\times\Gamma},\,m=1,\cdots,M\big\}, \\ \nonumber
			\tilde{X}_k(\Gamma):&=\big\{v_k\in L^2(I;L^2(\Gamma)):\,v_k|_{I_m}\in P_0(I_m;L^2(\Gamma)),\,m=1,\cdots,M\big\}.
		\end{split}
	\end{equation}
	
	When $\Omega$ is a polytope, we introduce a family of quasi-uniform and shape regular partitions $\mathscr{T}_h=\{\tau\}$ in the sense of Ciarlet \cite{Ciarlet}, where $\tau$ is the $N$-simplex with diameter $h_\tau$. We denote the mesh size of $\mathscr{T}_h$ by $h:=\max \limits_{\tau\in\mathscr{T}_h}h_\tau$. Define the $P_1$ finite element space
	\begin{equation}
		V_h:=\big\{v_h\in C(\bar{\Omega}):\,v_h|_\tau\in P_1(\tau),\ \forall\tau\in\mathscr{T}_h\big \},
	\end{equation}
where $P_1(\tau)$ denotes the space of linear functions in $\tau$. We set $V^0_h:=V_h\cap H^1_0(\Omega)$ and
		\begin{equation}
			V_h(\Gamma):=\big\{v_h|_\Gamma:\,v_h\in V_h\big\}.\nonumber
	\end{equation}

	Let $\pi_h:C(\bar\Omega)\rightarrow V_h$ be the Lagrange interpolation operator (cf. \cite{BrennerScott2008}) or the Cl\'ement interpolation operator from $L^1(\Omega)$ to $V_h$ (\cite{Clement}). Let $P_h:L^2(\Omega)\to V_h$ be the $L^2$-projection operator defined as follows: For any $y\in L^2(\Omega),\,P_hy$ satisfies
	\begin{equation}
		(P_h y,v_h)=(y,v_h)\qquad \forall v_h\in V_h.\nonumber
	\end{equation}
	Assume that $R_h:H^1_0(\Omega)\rightarrow V^0_h$ is the Ritz projection operator, i.e., for any $\varphi\in H^1_0(\Omega)$, $R_h\varphi\in V_h^0$ satisfies
	\begin{equation}
		(\nabla R_h \varphi,\nabla v_h)=(\nabla\varphi,\nabla v_h)\qquad\forall v_h\in V^0_h.\nonumber
	\end{equation}

	The following lemma  gives  the inverse estimate and error estimates  for the Ritz projection (cf. \cite{BrennerScott2008}), which will be frequently used in this article.
	\begin{Lemma}\label{RitePE}
		Let $R_h$ be the Ritz projection operator, then the following estimates hold:
		\begin{align*}
			&\| v- R_hv\|_{H^1(\Omega)}\le Ch\| \nabla^2v\|_{L^2(\Omega)}&&\forall v\in H^1_0(\Omega)\cap H^2(\Omega),\\
			&\| v- R_hv\|_{L^2(\Omega)} \le Ch\| \nabla(v- R_hv)\|_{L^2(\Omega)} &&\forall v\in H^1_0(\Omega).
		\end{align*}
	\end{Lemma}
	\begin{Remark}
		Although the operators $R_h$, $P_h$ and $\pi_h$ are defined in the spaces of functions independent of time, it is  possible to extend their definitions to the time-dependent case that have to be understood pointwise in time. Then Lemma \ref{RitePE} is also valid for the time-dependent case with corresponding estimates under time-space norms.
	\end{Remark}
	
	In order to define the DG(0)-CG(1) discrete scheme for parabolic equations, we introduce the following time-space finite element spaces:
		\begin{equation}
			\begin{split}
				X^{0}_{kh}:=&\big\{v_{kh}\in X^0_k:\ v_{kh}|_{I_m}\in P_0(I_m;\,V^0_h),\ m=1,\cdots,M\big\},\\\nonumber
				X_{kh}:=&\big\{v_{kh}\in X_k:\ v_{kh}|_{I_m}\in P_0(I_m;V_h),\ m=1,\cdots,M\big\},\nonumber
			\end{split}
		\end{equation}
		where $P_0(I_m;V_h)$ $(m=1,\cdots,M)$ denotes the function space of constants defined in $I_m$ and valued in $V_h$, and the definition of $P_0(I_m;V^0_h)$ is similar. In addition, we define
		\begin{equation}
			X_{kh}(\Gamma):=\big\{v_{kh}\in X_k(\Gamma):\ v_{kh}|_{I_m}\in P_0(I_m;V_h(\Gamma)),\,m=1,\cdots,M\big\},\nonumber
		\end{equation}
		where $P_0(I_m;V_h(\Gamma))$ is the function space of constants defined in $I_m$ and valued in $V_h(\Gamma)$.
	
	For the definition of the temporal DG(0) scheme for parabolic equations, we need the following notations: for any $v_k\in X_k$,
		\begin{align*}
			&v_{k,m}=v^-_{k,m}:=\lim_{t\to 0^+}v_k(t_m-t),\quad v_{k,m+1}=v^+_{k,m}:=\lim_{t\to 0^+}v_k(t_m+t),\\
			& [v_k]_m:=v^+_{k,m}-v^-_{k,m},\quad m=1,\cdots,M,\nonumber
		\end{align*}
		where $v_{k,m}:=v_k|_{I_m}$ and $[v_k]_m$ denotes the jump of $v_k$ at nodes $t_m$.
		
		Given two piecewise-in-time $H^1$-functions  $v,\,w\in L^2(I;H^1(\Omega))$, we define the bilinear form $B:(v,w)\to \mathbb{R}$ as follows:
		\begin{equation}\label{bilinearF}
			B(v,w):=\sum_{m=1}^M( \partial_tv,w )_{I_m}+(\nabla v,\nabla w)_I+\sum_{m=2}^M([v]_{m-1},w^+_{m-1})+(v^+_0,w^+_0),
		\end{equation}
		where $(\cdot,\cdot)_{I_m}$ denotes the inner product in space $L^2(I_m;L^2(\Omega))$ (cf.      \cite{MeidnerVexler}). Applying  integration by parts to the above bilinear form $B$, we obtain the following dual representation:
		\begin{equation}\label{dbilinearF}
			B(v,w)=-\sum_{m=1}^M( v,\partial_tw )_{I_m}+(\nabla v,\nabla w)_I-\sum_{m=1}^{M-1}(v^-_m,[w]_m)+(v^-_M,w^-_M).
		\end{equation}
Note that if $v,w\in X_k$, then the first term in $B$ vanishes. Similarly, if $v,w$ are continuous functions in time, then the terms of $B$ containing jumps also vanish.
	
In the following two subsections, we consider the discretization of the optimal control problem (\ref{BC_Functioal}), where the state variable is approximated by finite elements, and the control variable is discretized by the variational discretization (cf. \cite{Michael}). In subsection 4.2, we apply the DG(0) method for the temporal semi-discretization of the state variable for problems posed on   bounded polytopal/smooth  domains. In subsection 4.3, we use the DG(0)-CG(1) method, i.e. piecewise constant functions in time and continuous piecewise linear polynomials in space, to discretize the state variable for problems posed on polytopal domains.

\subsection{Discretization in time}

In order to give the temporal semi-discrete scheme of the parabolic equation (\ref{BC_stateE}), we define the boundary $L^2$-projection operator in time $\tilde{P}_k:L^2(I;L^2(\Gamma))\to \tilde{X}_k(\Gamma)$, such that for any $w\in L^2(I;L^2(\Gamma))$, $\tilde{P}_kw\in \tilde{X}_k(\Gamma)$ satisfies
\begin{equation}\label{pro_boundary_condition}
	\tilde{P}_kw|_{I_m}:=\frac{1}{k_m}\int_{t_{m-1}}^{t_m}w(s)ds,\qquad m=1,\cdots,M.
\end{equation}
For simplicity, we set $\tilde{P}^m_ku:=\tilde{P}_ku|_{I_m}$.

The semi-discrete parabolic Dirichlet boundary control problem is given by
\begin{equation}\label{time_semi_control}
	\mathop{\rm min}\limits_{u\in U_{ad},\ y_k(u)\in X_k}J_k(y_k(u),u)=\frac{1}{2}\Vert y_k(u)-y_d\Vert^2_I+\frac{\alpha}{2}\Vert u\Vert^2_{L^2(I;L^2(\Gamma))},
\end{equation}
where $y_k(u)$ is the semi-discrete state variable satisfying the following equation:
\begin{equation}\label{time_semi_state_e}
		B(y_k(u),\varphi_k)=(f,\varphi_k)_I+(y_0,\varphi^+_{k,0})\qquad\forall\varphi_k\in X^0_k,\quad	y_k(u)|_{I\times\Gamma}=\tilde{P}_ku.
\end{equation}
Note that in the   discrete problem \eqref{time_semi_control} the control variable is not explicitly discretized, which is the so-called variational discretization proposed in \cite{Michael} for optimal control problems.

For any $u\in U_{ad}$, the semi-discrete scheme (\ref{time_semi_state_e}) admits a unique solution $y_k(u)\in X_k$. The control to discrete state mapping $S_k:L^2(I;L^2(\Gamma))\to X_k$ is defined by $S_ku:=y_k(u)$ for any $u\in L^2(I;L^2(\Gamma))$ and is Frech\'et differentiable. Then we obtain the  reduced optimization problem
\begin{equation}\label{time_semi_optimal}
	\mathop{\rm min}\limits_{u\in U_{ad}}\hat{J}_k(u):=J(S_ku,u).
\end{equation}
By standard arguments, we can prove  that the semi-discrete optimal control problem \eqref{time_semi_optimal} admits a unique solution $\bar{u}_k$ (cf. \cite[Page 53]{HinzePinnauUlbrich}).
Moreover, the reduced cost functional $\hat{J}_k$ is also Frech\'et differentiable and
\begin{equation}\label{semi-cost_derivative}
	\hat{J}^{'}_k(u)v=\int_{\Omega_T}(y_k(u)-y_d)\tilde{y}_k(v)dxdt+\alpha\int_{\Sigma_T}uvdsdt\qquad\forall v\in L^2(I;L^2(\Gamma)),
\end{equation}
where $\tilde{y}_k(v)\in X_k$ is the solution of the following semi-discrete problem:
\begin{equation}\label{time_semi_state_ee}
		B(\tilde{y}_k(v),\varphi_k)=0\qquad\forall\varphi_k\in X^0_k,\quad
		\tilde{y}_k(v)|_{I\times\Gamma}=\tilde{P}_kv.
\end{equation}

In the following theorem, we give the first order optimality condition for  the semi-discrete optimal control problem \eqref{time_semi_optimal}.
\begin{Theorem}\label{time_discrete_NC}
The pair $(\bar{u}_k,\bar{y}_k)$  is the optimal solution of the semi-discrete  control problem \eqref{time_semi_control} if and only if $\bar{y}_k:=S_k\bar{u}_k=y_k(\bar{u}_k)$ and the following first order optimality condition holds:
	\begin{equation}\label{semi_discrete_VI}
		\hat{J}^{'}_k(\bar{u}_k)(v-\bar{u}_k)\ge0\qquad \forall v\in U_{ad}.
	\end{equation}
Furthermore, there exists a semi-discrete adjoint state variable $\bar{z}_k\in X_k^0$ defined by
	\begin{equation}\label{time_discrete_adjoint}
		B(\varphi_k,\bar{z}_k)=(\bar{y}_k-y_d,\varphi_k)_I\qquad \forall\varphi_k\in X_k^0,
	\end{equation}
such that \eqref{semi_discrete_VI} can be equivalently written as
	\begin{equation}\label{semi_discrete_VI_1}
		\hat{J}^{'}_k(\bar{u}_k)(v-\bar{u}_k)=\int_{\Sigma_T}(\alpha \bar{u}_k-\partial_n\bar{z}_k)(v-\bar{u}_k)dsdt\ge0\qquad \forall v\in U_{ad},
	\end{equation}
	where $\partial_n\bar{z}_k\in X_k(\Gamma)$ is the outward normal derivative of $\bar{z}_k$  defined by
	\begin{equation}\label{semi_discrete_norm_deriv}
		\int_{\Sigma_T}\partial_n\bar{z}_k\phi_kdsdt=-\int_{\Omega_T}(\bar{y}_k-y_d)p_k(\phi_k)dxdt\qquad \forall\phi_k\in X_k(\Gamma),
	\end{equation}
	and $p_k(\phi_k)\in X_k$ is the solution of the following semi-discrete scheme:
	\begin{equation}\label{semi_discrete_VI_2}
			B(p_k(\phi_k),\varphi_k)=0\qquad \forall\varphi_k\in X^0_k,\quad
			p_k(\phi_k)|_{I\times\Gamma}=\phi_k.
	\end{equation}
\end{Theorem}
\begin{proof}
	Since the optimization problem (\ref{time_semi_control}) is strictly convex, the first order optimality condition  (\ref{semi_discrete_VI}) is a direct consequence of calculus of variations. Therefore, we only need to verify \eqref{semi_discrete_VI_1}. For any given $v\in U_{ad}$, let $\phi_k:=\tilde{P}_k(v-\bar{u}_k)$ in \eqref{semi_discrete_VI_2}, the superposition principle of linear equations implies that $p_k(\tilde{P}_k(v-\bar{u}_k))=y_k(v)-\bar{y}_k$ satisfies \eqref{semi_discrete_VI_2}. Then it follows from  (\ref{semi_discrete_norm_deriv}) that
	\begin{align*}
		\hat{J}^{'}_k(\bar{u}_k)(v-\bar{u}_k)&=\int_{\Omega_T}(\bar{y}_k-y_d)(y_k(v)-\bar{y}_k)dxdt+\alpha\int_{\Sigma_T}\bar{u}_k(v-\bar{u}_k)dsdt\\
		&=\int_{\Omega_T}(\bar{y}_k-y_d)p_k(\tilde{P}_k(v-\bar{u}_k))dxdt+\alpha\int_{\Sigma_T}\bar{u}_k(v-\bar{u}_k)dsdt\\
		&=-\int_{\Sigma_T}\partial_n\bar{z}_k\tilde{P}_k(v-\bar{u}_k)dsdt+\alpha\int_{\Sigma_T}\bar{u}_k(v-\bar{u}_k)dsdt\\
		&=-\int_{\Sigma_T}\partial_n\bar{z}_k(v-\bar{u}_k)dsdt+\alpha\int_{\Sigma_T}\bar{u}_k(v-\bar{u}_k)dsdt\\
		&=\int_{\Sigma_T}(\alpha \bar{u}_k-\partial_n\bar{z}_k)(v-\bar{u}_k)dsdt	\ge 0.
	\end{align*}
	This completes the proof.
\end{proof}
\begin{Remark}
In fact, one can easily check that the definition of the outward normal derivative in \eqref{semi_discrete_norm_deriv}   is equivalent to the usual one.
\end{Remark}

\subsection{Discretization in space}
When $\Omega$ is a polytope, we can further consider the time-space discretization of the optimal control problem \eqref{BC_Functioal} by using the DG(0)-CG(1) method. To this end, we define a family of $L^2$-projections.
	\begin{Definition}\label{defn_pprojection}
		Let $P_h:L^2(I;L^2(\Omega))\to L^2(I;V_h)$ and $P_{kh}:L^2(I;L^2(\Omega))\to X_{kh}$ be two orthogonal projection operators such that for any $w\in L^2(I;L^2(\Omega))$, $P_hw$ and $P_{kh}w$ satisfy
		\begin{align*}
			&(w-P_hw,v_h)_{L^2(I;L^2(\Omega))}=0\qquad\forall v_h\in L^2(I;V_h), \\
			&(w-P_{kh}w,v_{kh})_{L^2(I;L^2(\Omega))}=0\qquad\forall v_{kh}\in X_{kh},
		\end{align*}
		respectively, where $L^2(I;V_h)$ denotes the  space of functions defined in $I$ and valued in $V_h$.
		
		Similarly, we can define another two projections $\tilde{P}_h:L^2(I;L^2(\Gamma))\to L^2(I;V_h(\Gamma))$ and $\,\tilde{P}_{kh}:L^2(I;L^2(\Gamma))\to X_{kh}(\Gamma)$ with  $L^2(I;V_h(\Gamma))$ and $X_{kh}(\Gamma)$  playing the role of $L^2(I;V_h)$ and $X_{kh}$, respectively.
\end{Definition}

The time-space discretization of the optimal control problem \eqref{BC_Functioal} is given by
\begin{equation}\label{space_semi_control}
	\mathop{\rm min}\limits_{u\in U_{ad},\ y_{kh}(u)\in X_{kh}} J_{kh}(y_{kh}(u),u)=\frac{1}{2}\Vert y_{kh}(u)-y_d\Vert^2_I+\frac{\alpha}{2}\Vert u\Vert^2_{L^2(I;L^2(\Gamma))},
\end{equation}
where $y_{kh}(u)\in X_{kh}$ is the discrete state variable satisfying the following discrete state equation:
\begin{equation}\label{spatial_dicrete_state}
		B(y_{kh}(u),\varphi_{kh})=(f,\varphi_{kh})_I+(y_0,\varphi^+_{kh,0}) \qquad\forall\varphi_{kh}\in X^0_{kh},\quad
		y_{kh}(u)|_{I\times\Gamma}=\tilde{P}_{kh}u.
\end{equation}
Again, the control variable is not explicitly discretized in the above discrete control problem \eqref{space_semi_control}. However, the discrete adjoint state will yield an implicit discretization of the control.

For any given $u\in U_{ad}$, we can prove that the discrete state equation (\ref{spatial_dicrete_state}) has a unique solution $y_{kh}(u)\in X_{kh}$. Then we   define a discrete control-to-state linear operator $S_{kh}:L^2(I;L^2(\Gamma))\to X_{kh}$ by $S_{kh}u:=y_{kh}(u)$ for any $u\in L^2(I;L^2(\Gamma))$. We are led to the following reduced optimization problem:
 \begin{equation}\label{reduecd_optimal_problem}
 	\mathop{\rm min}\limits_{u\in U_{ad}}\hat{J}_{kh}(u):=J(S_{kh}u,u).
 \end{equation}
It is easy to check that the above discrete optimization problem has a unique solution, denoted by $\bar{u}_{kh}$. The first order Frech\'et derivative of $J_{kh}$ at $u\in L^2(I;L^2(\Gamma))$ can be calculated as
\begin{equation}
	\hat{J}^{'}_{kh}(u)v=\int_{\Omega_T}(y_{kh}(u)-y_d)\tilde{y}_{kh}(v)dxdt+\alpha\int_{\Sigma_T}uvdsdt\qquad\forall v\in U_{ad},\nonumber
\end{equation}
where $\tilde{y}_{kh}(v)\in X_{kh}$ is the solution of the following equation:
$$
	B(\tilde{y}_{kh}(v),\varphi_{kh})=0\quad\ \ \ \forall\varphi_{kh}\in    X_{kh}^0,\quad
	\tilde{y}_{kh}(v)|_{I\times\Gamma}=\tilde{P}_{kh}v.
$$

To simplify the Frech\'et derivative of $\hat{J}_{kh}$, we first define the discrete outward normal derivative for the DG(0)-CG(1) finite element solution of the backward parabolic equation.

\begin{Definition}\label{full_discrete_noral_div}
	For any $g\in L^2(I;L^2(\Omega))$, let $z$ be the solution of the  backward parabolic equation \eqref{T_back_eq}.
	Let $z_{kh}\in X^0_{kh}$ be the DG(0)-CG(1) finite element solution of  \eqref{T_back_eq} defined by
	\begin{equation}\label{full_discrete_z}
		B(\varphi_{kh},z_{kh})=(g,\varphi_{kh}) \qquad\forall\varphi_{kh}\in X^0_{kh}.
	\end{equation}
	Then the time-space discrete normal derivative $\partial^h_n z_{kh}\in X_{kh}(\Gamma)$ of $z_{kh}$ on $\Gamma$ is defined by
	\begin{equation}\label{full_discrete_norm_div}
		\int_{\Sigma_T}\partial^h_n z_{kh}\phi_{kh}dsdt=-\int_{\Omega_T}gp_{kh}(\phi_{kh})dxdt \qquad\forall\phi_{kh}\in X_{kh}(\Gamma),
	\end{equation}
	where $p_{kh}(\phi_{kh})\in X_{kh}$ satisfies the following equation for given $\phi_{kh}\in X_{kh}(\Gamma)$:
	\begin{equation}\label{Test_EQ}
			B(p_{kh}(\phi_{kh}),\varphi_{kh})=0 \qquad \forall\varphi_{kh}\in X^0_{kh},\quad
			p_{kh}(\phi_{kh})|_{I\times\Gamma}=\phi_{kh}.
	\end{equation}
\end{Definition}
The above definition is the fully discrete version of Definition \ref{def_T} by replacing $f,\,y_0$ with zeros. 
We also refer to \cite{WM-Hinze} for another definition of discrete normal derivatives.
\begin{Definition}\label{defn_Norm}
	Let $z$ be the solution of equation \eqref{T_back_eq} and $z_{kh}$ be the corresponding time-space fully discrete solution. The discrete function $\partial^h_n z_{kh}\in X_{kh}(\Gamma)$ is called the discrete normal derivative of $z_{kh}$ on $\Gamma$, if $\partial^h_n z_{kh}$ satisfies
	\begin{equation}\label{paper}
		\int_{\Sigma_T}\partial^h_n z_{kh}\Phi_{kh} dsdt=B(\Phi_{kh},z_{kh})-\int_{\Omega_T}g\Phi_{kh} dxdt\qquad\forall\,\Phi_{kh}\in X_{kh}.
	\end{equation}
\end{Definition}
In fact, Definitions \ref{full_discrete_noral_div} and \ref{defn_Norm} are equivalent, which is verified in the following proposition.
\begin{Proposition}
	Definitions \ref{full_discrete_noral_div} and \ref{defn_Norm} are equivalent.
\end{Proposition}
\begin{proof}
	Firstly, we prove that $\partial^h_n z_{kh}$ satisfying the equality \eqref{full_discrete_norm_div} in Definition \ref {full_discrete_noral_div}  implies \eqref{paper} in Definition \ref{defn_Norm}. For any $\Phi_{kh}\in X_{kh}$, we define $\Psi_{kh}\in X^0_{kh}$ by
	\begin{equation}\label{psi}
		B(\Psi_{kh},\varphi_{kh})=B(\Phi_{kh},\varphi_{kh})\qquad\forall\varphi_{kh}\in X^0_{kh}.
	\end{equation}
	Let $\Theta_{kh}:=\Phi_{kh}-\Psi_{kh}$, then $\Theta_{kh}$ is the solution of
	\begin{equation}\nonumber
			B(\Theta_{kh},\varphi_{kh})=0 \qquad\forall\varphi_{kh}\in X^0_{kh},\quad
			\Theta_{kh}|_{I\times\Gamma}=\Phi_{kh}|_{I\times\Gamma}.
	\end{equation}
	Setting $\phi_{kh}=\Phi_{kh}|_{I\times\Gamma}$ in the equality \eqref{full_discrete_norm_div},  there holds the following equality:
	\begin{equation}
		\int_{\Sigma_T}\partial^h_n z_{kh}\Phi_{kh} dsdt=-\int_{\Omega_T}g\Theta_{kh} dxdt=-\int_{\Omega_T}g\Phi_{kh} dxdt+\int_{\Omega_T}g\Psi_{kh} dxdt.\nonumber
	\end{equation}
	Note that $z_{kh}$ is the solution of \eqref{full_discrete_z},  it follows from  \eqref{psi} that
	\begin{equation}
		\int_{\Sigma_T}\partial^h_n z_{kh}\Phi_{kh} dsdt=B(\Phi_{kh},z_{kh})-\int_{\Omega_T}g\Phi_{kh} dxdt,\nonumber
	\end{equation}
	which implies that $\partial^h_n z_{kh}$ satisfies Definition \ref{defn_Norm}. Reversely, $\partial^h_n z_{kh}$ given in Definition \ref{defn_Norm} satisfies the identity \eqref{full_discrete_norm_div} in Definition \ref{full_discrete_noral_div}. This completes the proof.
\end{proof}

Now we are ready to derive the discrete first order optimality system.
\begin{Theorem}\label{full_discrere_control}
	The pair $(\bar{u}_{kh},\bar{y}_{kh})\in U_{ad}\times X_{kh}$ is the optimal solution of the fully discrete optimal control problem \eqref{reduecd_optimal_problem} if and only if $\bar{y}_{kh}:=S_{kh}\bar{u}_{kh}=y_{kh}(\bar{u}_{kh})$ and
	\begin{equation}\label{optimal_condition3}
		\hat{J}^{'}_{kh}(\bar{u}_{kh})(v-\bar{u}_{kh})\ge0 \qquad\forall v\in U_{ad}.
	\end{equation}
	Furthermore, we define the fully discrete adjoint state $\bar{z}_{kh}\in X^0_{kh}$ as
	\begin{equation}
		B(\varphi_{kh},\bar{z}_{kh})=(\bar{y}_{kh}-y_d,\varphi_{kh})\qquad\forall\varphi_{kh}\in X^0_{kh},
	\end{equation}
	with which the optimality condition \eqref{optimal_condition3} can be equivalently written as
	\begin{equation}\label{optimal_condition_adjoint}
		\hat{J}^{'}_{kh}(\bar{u}_{kh})(v-\bar{u}_{kh})=\!\!\int_{\Sigma_T}\!\!\!\!(\alpha\bar{u}_{kh}-\partial^h_n\bar{z}_{kh})(v-\bar{u}_{kh})dsdt\ge0 \qquad\forall v\in U_{ad},
	\end{equation}
	where $\partial^h_n\bar{z}_{kh}\in X_{kh}(\Gamma)$ is the discrete normal derivative of  $\bar{z}_{kh}\in X^0_{kh}$ defined in Definition \ref{full_discrete_noral_div}.
\end{Theorem}

	\section{Error estimates for the state equation}
	\setcounter{equation}{0}
	In order to perform the error analysis  for the discrete optimal control problem, we shall first provide a priori error estimates for the finite element discretization of the state equation.

	\subsection{Analysis of the temporal discretization error}
	In this subsection we mainly   estimate the error  between the temporal semi-discretization $y_k(u)$ and $y:=y(u)$ for given $u$. By using the Aubin-Nitsche technique \cite{BrennerScott2008}, we convert the finite element error into the projection or interpolation error.
	
To this end, we  introduce some projections and interpolations. Define the $L^2$-projection $P_k:L^2(I;L^2(\Omega))\to \tilde{X}_k$ such that $P_kz$ satisfies
	\begin{equation}\label{pro_function_d}
		P_kz|_{I_m}:=\frac{1}{k_m}\int_{t_{m-1}}^{t_m}z(s)ds, \qquad m=1,\cdots,M,\quad\forall z\in L^2(I;L^2(\Omega)).
	\end{equation}
	For simplicity we write $P^m_kz:=P_kz|_{I_m},\, m=1,\cdots,M$. For any $0\le s\le1$, there holds
	\begin{equation}\label{pro_error_esti}
		\Vert z-P_kz\Vert_I\le Ck^s\Vert z\Vert_{H^s(I;L^2(\Omega))},\quad\forall z\in H^s(I;L^2(\Omega)).
	\end{equation}
	
	In addition, we   need the following two interpolation operators in time. Define the interpolation operators $\pi^r_k:C(\bar{I};L^2(\Omega))\to \tilde{X}_k$ and $\pi^l_k:C(\bar{I};L^2(\Omega))\to \tilde{X}_k$ such that for any $w\in C(\bar{I};L^2(\Omega))$, $\pi^r_kw$ and $\pi_k^lw$ satisfy
	\begin{equation}\label{def_interpolation}
			\pi^r_kw|_{I_m}=w(t_m),\quad
			\pi^l_kw|_{I_m}=w(t_{m-1}),\qquad m=1,\cdots,M.
	\end{equation}
	Note that the above interpolations are defined by taking the end point values on each subinterval, and the  interpolation error estimate is similar to \eqref{pro_error_esti}. Due to the fact $H^s(I;L^2(\Omega))\hookrightarrow C(\bar{I};L^2(\Omega))$ for $s\in (\frac{1}{2},1]$, we obtain for any $w\in H^s(I;L^2(\Omega))$ (cf. \cite{B.J.Li,Schtzau}) that
	\begin{equation}\label{interpolation_error_esti}
			\left\Vert w-\pi^l_kw\right\Vert_I\le Ck^s\Vert w\Vert_{H^s(I;L^2(\Omega))},\quad
			\left\Vert w-\pi^r_kw\right\Vert_I\le Ck^s\Vert w\Vert_{H^s(I;L^2(\Omega))}.
	\end{equation}

	The scheme \eqref{time_semi_state_e} is a temporal DG(0) discretization of equation \eqref{BC_stateE} (cf. \cite{V.Thom}), and there exists a unique  solution denoted by $y_k(u)\in X_k$. Note that $f\in L^2(I;L^2(\Omega)),\, y_0\in L^2(\Omega),\, u\in H^{\frac{1}{4}}(I;L^2(\Gamma))\cap L^2(I;H^{\frac{1}{2}}(\Gamma))$, then $$y\in H^{\frac{1}{2}}(I;L^2(\Omega))\cap L^2(I;H^1(\Omega)).$$ However, the DG(0) semi-discretization $y_k(u)\in X_k$ is piecewise constant in time and does not belong to $H^{\frac{1}{2}}(I;L^2(\Omega))$ although $\tilde{P}_ku\in H^{\frac{1}{4}}(I;L^2(\Gamma))\cap L^2(I;H^{\frac{1}{2}}(\Gamma))$. In fact, we can verify that (cf. \cite[Page 20]{J. Sokoowski})
	$$y_k(u)\in H^{\frac{1}{2}-\varepsilon}(I;L^2(\Omega))\quad \forall \varepsilon>0.$$ Therefore, the DG(0) semi-discretization \eqref{time_semi_state_e} is a nonconforming Galerkin method (i.e., $X_k\nsubseteq H^{\frac{1}{2}}(I;L^2(\Omega))$), but we still have the  Galerkin orthogonality property
	 \begin{equation}
		B(y-y_k(u),\varphi_k)=0\qquad\forall\varphi_k\in X^0_k.\nonumber
	 \end{equation}
		
	Below, let us give several lemmas on the stability and error estimation for the temporal semi-discretization of parabolic equations.
	\begin{Lemma}\label{lemma_g}
		For any $s\in (\frac{1}{2},1]$, assume that $y\in H^s(I;L^2(\Omega))\cap L^2(I;H^{2s}(\Omega))$ is the solution of \eqref{BC_stateE}. Define the piecewise constant function $g$ by
		\begin{align*}
			g|_{I_m}:=\frac{P^m_ky-y(t_m)}{k_m}-\frac{P^{m-1}_ky-y(t_{m-1})}{k_m},\qquad m=1,\cdots,M,\quad P^0_ky:=y(t_0)=y_0.
		\end{align*}
Then we have
		\begin{equation}
			\sum\limits_{m=1}^Mk_m^{1-(2s-1)}\Vert g\Vert^2_{L^2(I_m;L^2(\Omega))}\le C\Vert y\Vert^2_{H^s(I;L^2(\Omega))}.\nonumber
		\end{equation}
	\end{Lemma}
	
	\begin{proof}
	Using the condition \eqref{condition1} we  obtain
		\begin{equation}
			\sum\limits_{m=1}^Mk_m^{1-(2s-1)}\Vert g\Vert^2_{L^2(I_m;L^2(\Omega))}\le Ck^{1-(2s-1)}\sum\limits_{m=1}^M\Vert g\Vert^2_{L^2(I_m;L^2(\Omega))}.\nonumber
		\end{equation}
	Therefore,  we only need to estimate the right-hand side of the above inequality:
		\begin{align*}
			\sum\limits_{m=1}^M\Vert g\Vert^2_{L^2(I_m;L^2(\Omega))}&=\sum\limits_{m=1}^Mk_m^{-1}\Vert (P^m_ky-y(t_m))-(P^{m-1}_ky-y(t_{m-1}))\Vert^2_{L^2(\Omega)}\\
			&\le C\sum\limits_{m=1}^Mk_m^{-1}\Vert P^m_ky-y(t_m)\Vert^2_{L^2(\Omega)}\\
			&=C\sum\limits_{m=1}^Mk_m^{-1}\int_\Omega\Big(\frac{1}{k_m}\int_{t_{m-1}}^{t_m}y(s)ds-y(t_m)\Big)^2dx\\
			&=C\sum\limits_{m=1}^Mk_m^{-3}\int_\Omega\Big(\int_{t_{m-1}}^{t_m}(y(s)-y(t_m))ds\Big)^2dx\\
			&\le Ck^{-2}\sum\limits_{m=1}^M\int_\Omega\int_{t_{m-1}}^{t_m}\Big(y(s)-y(t_m)\Big)^2dsdx\\
			&=Ck^{-2}\Vert y-\pi^r_ky\Vert_I^2\\
			&\le Ck^{2(s-1)}\Vert y\Vert_{H^s(I;L^2(\Omega))}^2,
		\end{align*}
		where the second interpolation error estimate of \eqref{interpolation_error_esti} has been used. This completes the proof.
	\end{proof}

\begin{Lemma}\label{jumping_estimate}
	Let $y$ and $y_k(u)$ be the solution of \eqref{BC_stateE} and \eqref{time_semi_state_e},  respectively. For any  $y\in H^s(I;L^2(\Omega))\cap L^2(I;H^{2s}(\Omega))$ {\rm (}$s\in [\frac{1}{2}, 1] ${\rm )},  we have
	\begin{align}
		&\sum\limits_{m=1}^Mk_m^{1-(2s-1)}\Vert \nabla[P_ky-y_k(u)]_{m-1}\Vert^2_{L^2(\Omega)}
		+\sum\limits_{m=1}^Mk_m^{-(2s-1)}\Vert [P_ky-y_k(u)]_{m-1}\Vert^2_{L^2(\Omega)}\nonumber\\
		&\le C
		\begin{cases}
			\Vert y\Vert_{H^s(I;L^2(\Omega))}^2\quad \frac{1}{2}<s\le1,\\
			\Vert y\Vert_{H^{\frac{1}{2}}(I;L^2(\Omega))}^2+\left\Vert \partial_t y\right\Vert_{L^2(I;H^{-1}(\Omega))}^2\quad s=\frac{1}{2}.
		\end{cases}
	\end{align}
\end{Lemma}
\begin{proof}
	For each fixed $m=1,\cdots,M$, from  (\ref{BC_stateE}) it follows
	\begin{equation}
		(\nabla y,\nabla\varphi)_{I_m}+(\partial_t y,\varphi)_{I_m}=(f,\varphi)_{I_m}\qquad\forall\varphi\in P_0(I_m;H^1_0(\Omega)),\nonumber
	\end{equation}
	i.e.,
	\begin{equation}\label{cont_pro}
		(\nabla P_ky,\nabla\varphi)_{I_m}+(y(t_m)-y(t_{m-1}),\varphi)=(f,\varphi)_{I_m}\qquad\forall\varphi\in P_0(I_m;H^1_0(\Omega)),
	\end{equation}
	where $P_k$ is defined in \eqref{pro_function_d}. Similarly, the temporal semi-discretization solution $y_k(u)$ of (\ref{time_semi_state_e}) on the subinterval $I_m$ satisfies
	\begin{equation}\label{semi_step_pro}
		(\nabla y_k(u),\nabla\varphi)_{I_m}+([y_k(u)]_{m-1},\varphi)=(f,\varphi)_{I_m}\qquad\forall\varphi\in P_0(I_m;H^1_0(\Omega)),
	\end{equation}
where $[y_k(u)]_0:=y^+_{k,0}-y_0$. Subtracting (\ref{semi_step_pro}) from (\ref{cont_pro}), we obtain
	\begin{eqnarray*}
		(\nabla (P_ky-y_k(u)),\nabla\varphi)_{I_m}+(y(t_m)-y(t_{m-1})-[y_k(u)]_{m-1},\varphi)=0\qquad\forall\varphi\in P_0(I_m;H^1_0(\Omega)).
	\end{eqnarray*}
	Note that the solution $y$ of equation (\ref{BC_stateE}) belongs to $ H^{s}(I;L^2(\Omega))\cap L^2(I;H^{2s}(\Omega))$ $(\frac{1}{2}\le s)$ and $ H^1(I;H^{-1}(\Omega))\cap L^2(I;H^1(\Omega))$ (cf. \cite{J.L1}). Therefore, we can rewrite the above equality as
	\begin{equation}\label{semi_pro}
		\begin{split}		
		(\nabla (P_ky-y_k(u)),\nabla\varphi)_{I_m}+&([P_ky-y_k(u)]_{m-1},\varphi)\\
		&=
		\begin{cases}
			([P_ky]_{m-1},\varphi)-(y(t_m)-y(t_{m-1}),\varphi)\vspace{0.2cm}\quad\quad \frac{1}{2}<s\le 1,\\
			([P_ky]_{m-1},\varphi)-\langle\partial_t y,\varphi\rangle_{L^2(I_m;H^{-1},H^{1}_0)}\quad\quad s=\frac{1}{2}
		\end{cases}
	\end{split}
	\end{equation}
 for any $\varphi\in P_0(I_m;H^1_0(\Omega))$, where $[P_ky]_0:=P^1_ky-y_0$, $[P_ky-y_k(u)]_0:=P^1_ky-y_{k,1}(u)$, and $\langle\cdot,\cdot\rangle_{L^2(I_m;H^{-1},H^{1}_0)}$ denotes the duality pairing between $L^2(I_m;H^{-1}(\Omega))$ and $L^2(I_m;H^1_0(\Omega))$.

	Define the function $\tilde{g}$ as
	\begin{equation}
		\tilde{g}|_{I_m}:=\frac{P^m_ky-P^{m-1}_ky}{k_m}-\partial_t y|_{I_m}.\nonumber
	\end{equation}
	By using the function $g$ in Lemma \ref{lemma_g} and  the above function $\tilde{g}$, the equality (\ref{semi_pro}) becomes
	\begin{equation}\label{semi_pro1}
		(\nabla (P_ky-y_k(u)),\nabla\varphi)_{I_m}+([P_ky-y_k(u)]_{m-1},\varphi)=
		\begin{cases}
			(g,\varphi)_{I_m}\quad&\frac{1}{2}<s\le1,\\
			\langle\tilde{g},\varphi\rangle_{L^2(I_m;H^{-1},H^{1}_0)}\quad&s=\frac{1}{2}
		\end{cases}
	\end{equation}
for all $\varphi\in P_0(I_m;H^1_0(\Omega))$. By taking $\varphi|_{I_m}:=[P_ky-y_k(u)]_{m-1}$ in the above equality, we obtain
	\begin{equation}
		\begin{split}
			\big(\nabla (P_ky-y_k(u)),\nabla [P_ky-&y_k(u)]_{m-1}\big)_{I_m}+\left\Vert[P_ky-y_k(u)]_{m-1}\right\Vert^2_{L^2(\Omega)}\\
			&=
			\begin{cases}
				(g,[P_ky-y_k(u)]_{m-1})_{I_m}\quad&\frac{1}{2}<s\le1,\\
				\langle\tilde{g},[P_ky-y_k(u)]_{m-1}\rangle_{L^2(I_m;H^{-1},H^{1}_0)}\quad &s=\frac{1}{2},\nonumber
			\end{cases}
		\end{split}
	\end{equation}
	i.e.,
	\begin{equation}\label{pro_semi_jump}
		\begin{split}
			&\frac{1}{2}\Big(\Vert \nabla [P_ky-y_k(u)]_{m-1}\Vert^2_{I_m}+\Vert\nabla (P_ky-y_k(u))_m\Vert^2_{I_m}-\Vert\nabla (P_ky-y_k(u))_{m-1}\Vert^2_{I_m}\Big)\\&+\Vert[P_ky-y_k(u)]_{m-1}\Vert^2_{L^2(\Omega)}
			=
			\begin{cases}
				(g,[P_ky-y_k(u)]_{m-1})_{I_m}\quad&\frac{1}{2}<s\le1,\\
				\langle\tilde{g},[P_ky-y_k(u)]_{m-1}\rangle_{L^2(I_m;H^{-1},H^{1}_0)}\quad&s=\frac{1}{2},
			\end{cases}
		\end{split}
	\end{equation}
	where $(P_ky-y_k(u))_m:=P^m_ky-y_{k,m}(u)$ $(m\ne 0)$ and $(P_ky-y_k(u))_0:=0$. It remains to  estimate the right-hand side of the above equality for $\frac{1}{2}<s\le 1$ and $s=\frac{1}{2}$, respectively.
	
	For $s\in(\frac{1}{2},1]$, we use the first identity of \eqref{pro_semi_jump} and the Cauchy-Schwarz inequality to obtain
	\begin{align*}
		&\frac{1}{2}k_m\Big(\Vert \nabla [P_ky-y_k(u)]_{m-1}\Vert^2_{L^2(\Omega)}+\Vert\nabla (P_ky-y_k(u))_m\Vert^2_{L^2(\Omega)}-\Vert\nabla (P_ky-y_k(u))_{m-1}\Vert^2_{L^2(\Omega)}\Big)\\
		&+\Vert[P_ky-y_k(u)]_{m-1}\Vert^2_{L^2(\Omega)}\le\frac{k_m}{2}\Vert g\Vert^2_{I_m}+\frac{\Vert[P_ky-y_k(u)]_{m-1}\Vert^2_{I_m}}{2k_m}.\nonumber
	\end{align*}
Multiplying by $k^{-(2s-1)}_m$ on both sides of the above estimate and absorbing the second term of the right-hand side to the left, we obtain
	\begin{align*}
		&k^{1-(2s-1)}_m\Big(\Vert \nabla [P_ky-y_k(u)]_{m-1}\Vert^2_{L^2(\Omega)}+\Vert\nabla (P_ky-y_k(u))_m\Vert^2_{L^2(\Omega)}-\Vert\nabla (P_ky-y_k(u))_{m-1}\Vert^2_{L^2(\Omega)}\Big)\\
		&+k^{-(2s-1)}_m\Vert[P_ky-y_k(u)]_{m-1}\Vert^2_{L^2(\Omega)}
		\le k^{1-(2s-1)}_m\Vert g\Vert^2_{I_m}.
	\end{align*}
	By using the step size condition \eqref{condition2}, the above estimate can be rewritten as
	\begin{align*}
		&k^{1-(2s-1)}_m\Vert \nabla [P_ky-y_k(u)]_{m-1}\Vert^2_{L^2(\Omega)}+k^{1-(2s-1)}_m\Vert\nabla (P_ky-y_k(u))_m\Vert^2_{L^2(\Omega)}\\
		&-k^{1-(2s-1)}_{m-1}\Vert\nabla(P_ky-y_k(u))_{m-1}\Vert^2_{L^2(\Omega)}
		+k^{-(2s-1)}_m\Vert[P_ky-y_k(u)]_{m-1}\Vert^2_{L^2(\Omega)}
		\le k^{1-(2s-1)}_m\Vert g\Vert^2_{I_m},
	\end{align*}
 where $k_0:=1$. Summing  up the above estimate for $m$ and using Lemma \ref{lemma_g}, we obtain the result for the case $s\in (\frac{1}{2}, 1]$.
	
	When $s=\frac{1}{2}$, we use the second identity of (\ref{pro_semi_jump}) and the Cauchy-Schwarz inequality to obtain
	\begin{align*}
		&\frac{1}{2}k_m\Big(\left\Vert\nabla [P_ky-y_k(u)]_{m-1}\right\Vert^2_{L^2(\Omega)}+\left\Vert\nabla (P_ky-y_k(u))_m\right\Vert^2_{L^2(\Omega)}-\left\Vert\nabla (P_ky-y_k(u))_{m-1}\right\Vert^2_{L^2(\Omega)}\Big)\\
		&+\left\Vert[P_ky-y_k(u)]_{m-1}\right\Vert^2_{L^2(\Omega)}
		\le\frac{k_m}{2}\left\Vert \frac{P^m_ky-P^{m-1}_ky}{k_m}\right\Vert^2_{I_m}+\frac{\left\Vert[P_ky-y_k(u)]_{m-1}\right\Vert^2_{I_m}}{2k_m}\\&+C\left\Vert\partial_t y\right\Vert^2_{L^2(I_m;H^{-1}(\Omega))}
		+\frac{1}{4}\left\Vert \nabla[P_ky-y_k(u)]_{m-1}\right\Vert^2_{I_m}.\nonumber
	\end{align*}
Absorbing the second and fourth terms of the right-hand side to the left and employing the step size condition \eqref{condition2} to the left, we have
	\begin{align*}
		&\frac{1}{2}k_m\left\Vert \nabla [P_ky-y_k(u)]_{m-1}\right\Vert^2_{L^2(\Omega)}+\!k_m\left\Vert\nabla (P_ky-y_k(u))_m\right\Vert^2_{L^2(\Omega)}-k_{m-1}\left\Vert\nabla (P_ky-y_k(u))_{m-1}\right\Vert^2_{L^2(\Omega)}\\
		&+\Vert[P_ky-y_k(u)]_{m-1}\Vert^2_{L^2(\Omega)}
		\le k_m\left\Vert \frac{P^m_ky-P^{m-1}_ky}{k_m}\right\Vert^2_{I_m}+C\left\Vert\partial_t y\right\Vert^2_{L^2(I_m;H^{-1}(\Omega))}.
	\end{align*}
	Summing up for $m=1,\cdots,M$ in the above estimate, we arrive at
	\begin{equation}\label{desire_E}
		\begin{split}
			&\sum\limits_{m=1}^M\Big( \frac{1}{2}k_m\left\Vert \nabla [P_ky-y_k(u)]_{m-1}\right\Vert^2_{L^2(\Omega)}+\left\Vert[P_ky-y_k(u)]_{m-1}\right\Vert^2_{L^2(\Omega)}\Big)\\
			&\le\sum\limits_{m=1}^M \Vert P^m_ky-P^{m-1}_ky\Vert^2_{L^2(\Omega)}+C\left\Vert\partial_t y\right\Vert^2_{L^2(I;H^{-1}(\Omega))}.
		\end{split}
	\end{equation}
	Hence, it remains to estimate the first term on the right-hand side of (\ref{desire_E}):
	\begin{align}\label{pro_jum_est}
		&\sum\limits_{m=1}^M \left\Vert P^m_ky-P^{m-1}_ky\right\Vert^2_{L^2(\Omega)}\notag
		=\sum\limits_{m=1}^M \left\Vert \frac{1}{k_m}\int_{t_{m-1}}^{t_m}y(s)ds-\frac{1}{k_{m-1}}\int_{t_{m-2}}^{t_{m-1}}y(s)ds\right\Vert^2_{L^2(\Omega)} \notag \\
		&=\sum\limits_{m=1}^M \left\Vert \frac{1}{k_mk_{m-1}}\int_{t_{m-1}}^{t_m}\int_{t_{m-2}}^{t_{m-1}}(y(s)-y(\tau))dsd\tau\right\Vert^2_{L^2(\Omega)}\notag \\
		&\le\sum\limits_{m=1}^M \frac{1}{k_m^2k^2_{m-1}}\int_{t_{m-1}}^{t_m}\int_{t_{m-2}}^{t_{m-1}}\frac{\left\Vert y(s)-y(\tau)\right\Vert^2_{L^2(\Omega)}}{(s-\tau)^2}dsd\tau\int_{t_{m-1}}^{t_m}\int_{t_{m-2}}^{t_{m-1}}(s-\tau)^2dsd\tau\notag\\
		&\le \sum\limits_{m=1}^M \frac{1}{k_m^2k^2_{m-1}}k_mk_{m-1}(k_m+k_{m-1})^2\int_{t_{m-1}}^{t_m}\int_{t_{m-2}}^{t_{m-1}}\frac{\left\Vert y(s)-y(\tau)\right\Vert^2_{L^2(\Omega)}}{(s-\tau)^2}dsd\tau\notag \\
		&\le C \sum\limits_{m=1}^M\int_{t_{m-1}}^{t_m}\int_{t_{m-2}}^{t_{m-1}}\frac{\left\Vert y(s)-y(\tau)\right\Vert^2_{L^2(\Omega)}}{(s-\tau)^2}dsd\tau\notag \\
		&\le C\Vert y\Vert^2_{H^{\frac{1}{2}}(I;L^2(\Omega))},
	\end{align}
	where we have used \eqref{condition1} and the norm definition of the space $H^{\frac{1}{2}}(I;L^2(\Omega))$. This completes the proof of the lemma.
\end{proof}

Based on the previous lemma, we obtain the following result.
\begin{Lemma}\label{jump_main_esti}
	Assume that $y$ and $y_k(u)$ are the solutions of equations \eqref{BC_stateE} and \eqref{time_semi_state_e} respectively. For $y\in H^s(I;L^2(\Omega))\cap L^2(I;H^{2s}(\Omega))$, $\forall s\in [\frac{1}{2},1]$, there holds
	\begin{align}
&\sum\limits_{m=1}^Mk_m^{-(2s-1)}\Vert [y_k(u)]_{m-1}\Vert^2_{L^2(\Omega)}\le C
		\begin{cases}
			\left\Vert y\right\Vert_{H^s(I;L^2(\Omega))}^2\quad \frac{1}{2}<s\le1,\\
			\Vert y\Vert_{H^{\frac{1}{2}}(I;L^2(\Omega))}^2+\left\Vert \partial_t y\right\Vert_{L^2(I;H^{-1}(\Omega))}^2\quad s=\frac{1}{2}.\nonumber
		\end{cases}
	\end{align}
\end{Lemma}
\begin{proof}
	Note that
	\begin{align}\nonumber
		\sum\limits_{m=1}^Mk_m^{-(2s-1)}\Vert [y_k(u)]_{m-1}\Vert^2_{L^2(\Omega)}&\le 2\sum\limits_{m=1}^Mk_m^{-(2s-1)}\big(\Vert [P_ky-y_k(u)]_{m-1}\Vert^2_{L^2(\Omega)}
		+\Vert [P_ky]_{m-1}\Vert^2_{L^2(\Omega)}\big).\nonumber
	\end{align}
	By using Lemma \ref{jumping_estimate}, it suffices to estimate the second term on the right-hand side of the above estimate, which will be done by arguing the case $\frac{1}{2}<s\le1$ and $s=\frac{1}{2}$, respectively.
	
	For $s\in(\frac{1}{2},1]$, we apply the Cauchy-Schwarz inequality and the interpolation operators in \eqref{def_interpolation} to obtain
	\begin{eqnarray}\notag
		&&\sum\limits_{m=1}^Mk_m^{-(2s-1)}\Vert [P_ky]_{m-1}\Vert^2_{L^2(\Omega)}\notag\\
		&=&\sum\limits_{m=2}^Mk_m^{-(2s-1)}\left\Vert\frac{1}{k_m}\int^{t_m}_{t_{m-1}}y(s)ds-\frac{1}{k_{m-1}}\int^{t_{m-1}}_{t_{m-2}}y(s)ds\right\Vert^2_{L^2(\Omega)}+k^{-2s+1}_1\Vert P_k^1y-y_0\Vert^2_{L^2(\Omega)}\notag\\
		&=&\sum\limits_{m=2}^Mk_m^{-(2s-1)}\int_\Omega\Big(\frac{1}{k_m}\int^{t_m}_{t_{m-1}}y(s)ds-\frac{1}{k_{m-1}}\int^{t_{m-1}}_{t_{m-2}}y(s)ds\Big)^2dx+k^{-2s+1}_1\int_\Omega\big( P_k^1y-y_0\big)^2dx\notag\\
		&=&\sum\limits_{m=2}^Mk_m^{-(2s-1)}\int_\Omega\Big(\frac{1}{k_m}\int^{t_m}_{t_{m-1}}(y(s)-y(t_{m-1}))ds+\frac{1}{k_{m-1}}\int^{t_{m-1}}_{t_{m-2}}(y(t_{m-1})-y(s))ds\Big)^2dx\notag\\
		&&+k^{-2s+1}_1\int_\Omega\Big( \frac{1}{k_1}\int_{t_0}^{t_1}(y(s)-y_0)ds\Big)^2dx\notag\\
		&\le& 2\sum\limits_{m=2}^Mk_m^{-(2s-1)}\Big(\int_\Omega\Big(\frac{1}{k_m}\int^{t_m}_{t_{m-1}}(y(s)-y(t_{m-1}))ds\Big)^2dx\notag\\
		&&+\int_\Omega\Big(\frac{1}{k_{m-1}}\int^{t_{m-1}}_{t_{m-2}}(y(t_{m-1})-y(s))ds\Big)^2dx\Big)+k^{-2s+1}_1\int_\Omega\Big( \frac{1}{k_1}\int_{t_0}^{t_1}(y(s)-y_0)ds\Big)^2dx\notag\\
		&\le& C k^{-2s}\left( \sum\limits_{m=1}^M\int_\Omega\int^{t_m}_{t_{m-1}}\big(y(s)-y(t_{m-1})\big)^2dsdx+\sum\limits_{m=1}^M\int_\Omega\int^{t_m}_{t_{m-1}}\Big(y(s)-y(t_{m})\Big)^2dsdx\right)\notag\\
		&=&C k^{-2s}\big(\Vert y-\pi^l_ky\Vert_I^2+\Vert y-\pi^r_ky\Vert_I^2\big)\notag\\
		&\le& C\Vert y\Vert_{H^s(I;L^2(\Omega))}^2,\notag
	\end{eqnarray}
where we have used the interpolation error estimate \eqref{interpolation_error_esti}.
	
For $s=\frac{1}{2}$, we can directly use the estimate \eqref{pro_jum_est} in Lemma \ref{jumping_estimate} to obtain the result. This completes the proof.
\end{proof}

In the next theorem, we use the above lemmas and the Aubin-Nitsche trick to estimate the error of the DG(0) semi-discretization under the norm $\Vert\cdot\Vert_{L^2(I,L^2(\Omega))}$.

\begin{Theorem}\label{error_semi_d}
Let $y$ and $y_k(u)$ be the solution of equations \eqref{BC_stateE} and \eqref{time_semi_state_e}, respectively. For any $y\in H^s(I;L^2(\Omega))\cap L^2(I;H^{2s}(\Omega))$ ($s\in [\frac{1}{2}, 1]$) there holds
	\begin{align*}
		\Vert y-y_k(u)\Vert_{L^2(I;L^2(\Omega))}\le C\big(k^s\Vert& u\Vert_{H^{s-\frac{1}{4}}(I;L^2(\Gamma)}+k\Vert f\Vert_{L^2(I;L^2(\Omega))}\big)\\
		&+Ck^s
		\begin{cases}
			\Vert y\Vert_{H^s(I;L^2(\Omega))}\quad \frac{1}{2}<s\le1,\\
			\Vert y\Vert_{H^{\frac{1}{2}}(I;L^2(\Omega))}+\Vert \partial_t y\Vert_{L^2(I;H^{-1}(\Omega))}\quad s=\frac{1}{2}.\nonumber
		\end{cases}
	\end{align*}
\end{Theorem}
\begin{proof}
	Set $e_k:=y-y_k(u)$ and let $z\in L^2(I;H^2(\Omega)\cap H^1_0(\Omega))\cap H^1(I;L^2(\Omega))$ be the solution of (\ref{T_back_eq}) with $g:=e_k$.
We use the Aubin-Nitsche trick to estimate $\Vert e_k\Vert_{L^2(I,L^2(\Omega))}$. Note that
	\begin{align} \label{E1}\notag
		&\Vert e_k\Vert_{L^2(I;L^2(\Omega))}^2
		=(e_k,-\partial_tz-\Delta z)_I\\ 
		&=\int_{\Omega_T}(-\partial_tz-\Delta z)y dxdt+\int_{\Omega_T}y_k(u)(\partial_tz+\Delta z)dxdt\\ \notag
		&=\int_{\Sigma_T} (y_k(u)-u)\partial_nz dsdt+\int_{\Omega_T}fz dxdt+\int_\Omega y_0z(0)dx+\int_{\Omega_T}(y_k(u)\partial_tz-\nabla y_k(u)\cdot\nabla z) dxdt\\\notag
		&=\int_{\Sigma_T} (\tilde{P}_ku-u)\partial_nz dsdt+\int_{\Omega_T}fz dxdt+\int_\Omega y_0z(0)dx+\int_{\Omega_T}(y_k(u)\partial_tz-\nabla y_k(u)\cdot\nabla z) dxdt,
	\end{align}
    where we have used Definition \ref{def_T} and the equation (\ref{T_back_eq}). The  first term in the above estimate can be bounded by projection errors, so it remains to consider other terms. Note that
	\begin{eqnarray}\label{E2}\notag
		&&\int_{\Omega_T}(y_k(u)\partial_tz-\nabla y_k(u)\cdot\nabla z) dxdt
		=\sum\limits_{n=1}^M\int_{t_{n-1}}^{t_n}(y_k(u),\partial_tz)dt-(\nabla y_k(u),\nabla z)_I\\\notag
		&=&\sum\limits_{n=1}^M\big(y_{k,n}(u),z(t_n)-z(t_{n-1})\big)-(\nabla y_k(u),\nabla z)_I\\ \notag
		&=&\sum\limits_{n=2}^M\big(y_{k,n-1}(u),z(t_{n-1})\big)-\sum\limits_{n=2}^M\big(y_{k,n}(u),z(t_{n-1})\big)-(y^+_{k,0}(u),z(0))-(\nabla y_k(u),\nabla z)_I\\\notag
		&=&-\sum\limits_{n=2}^M\big([y_k(u)]_{n-1},z(t_{n-1})\big)-(y^+_{k,0}(u),z(0))-(\nabla y_k(u),\nabla z)_I\\\notag
		&=&-\sum\limits_{n=2}^M\big([y_k(u)]_{n-1},(P_kz)_{n-1}^+\big)-(\nabla y_k(u),\nabla P_kz)_I-(y^+_{k,0}(u),(P_kz)_0^+)\\\notag
		&&+\sum\limits_{n=2}^M\big([y_k(u)]_{n-1},(P_kz)_{n-1}^+-z(t_{n-1})\big) +(y^+_{k,0}(u),(P_kz)_0^+-z(0))\\  
		&=&-B(y_k(u),P_kz)+\sum\limits_{n=1}^M\big([y_k(u)]_{n-1},(P_kz)_{n-1}^+-z(t_{n-1})\big) +(y_0,(P_kz)_0^+-z(0)),
	\end{eqnarray}
	where we have used \eqref{bilinearF}, (\ref{time_semi_state_e}) and the definition of $P_k$. Combining the above estimates (\ref{E1}) and (\ref{E2}), we have
	\begin{align}\label{IE_3}
		\big\Vert e_k\big\Vert_{L^2(I;L^2(\Omega))}^2\notag
		&=\int_{\Sigma_T} \big(\tilde{P}_ku-u\big)\partial_nz dsdt+\int_{\Omega_T}f(z-P_kz )dxdt\\\notag
		&+ \sum\limits_{n=1}^M\big([y_k(u)]_{n-1},(P_kz)_{n-1}^+-z(t_{n-1})\big)\\
		&=:I_1+I_2+I_3.
	\end{align}
	Now we estimate the terms $I_1$, $I_2$ and $I_3$ in the above equality. Using the fact $u\in H^{s-\frac{1}{4}}(I;L^2(\Gamma))$ and the Cauchy-Schwarz inequality, we are led to
	\begin{align*}
		\vert I_1\vert&=\left\vert\int_{\Sigma_T} (\tilde{P}_ku-u)\partial_nz dsdt\right\vert\\
		&\le \Vert \tilde{P}_ku-u\Vert_{L^2(I;L^2(\Gamma))}\left\Vert \partial_nz-P_k\partial_nz\right\Vert_{L^2(I;L^2(\Gamma))}\\
		&\le Ck^s\Vert u\Vert_{H^{s-\frac{1}{4}}(I;L^2(\Gamma))}\Vert e_k\Vert_{I},
	\end{align*}
	where the following two estimates have been used:
	\begin{align*}
		\big\Vert \tilde{P}_ku-u\big\Vert_{L^2(I;L^2(\Gamma))}&\le Ck^{s-\frac{1}{4}}\Vert u\Vert_{H^{s-\frac{1}{4}}(I;L^2(\Gamma))},\quad
		\left\Vert \partial_nz-P_k\partial_nz\right\Vert_{L^2(I;L^2(\Gamma))}\le  Ck^{\frac{1}{4}}\Vert e_k\Vert_{I}.
	\end{align*}
	The first estimate is the usual projection error estimate (cf. \cite{BrennerScott2008}), while the second can be derived as follows. Since  $z\in H^1(I;L^2(\Omega))\cap L^2(I;H^2(\Omega))$, we have $\partial_nz\in H^{\frac{1}{4}}(I;L^2(\Gamma))$ and
	\begin{equation}
		\Vert \partial_nz\Vert_{H^{\frac{1}{4}}(I;L^2(\Gamma))}\le C(\Vert z\Vert_{L^2(I;H^2(\Omega))}+\Vert z\Vert_{H^1(I;L^2(\Omega))})\le C\Vert e_k\Vert_{I},\nonumber
	\end{equation}
where we have used the trace theorem \cite{J.L1} and the a priori estimate for parabolic equations \cite{L. C. EVANS}. Then we obtain
	\begin{align*}
		\left\Vert\partial_nz-P_k\partial_nz\right\Vert_{L^2(I;L^2(\Gamma))}\le  Ck^{\frac{1}{4}}\Vert \partial_nz\Vert_{H^{\frac{1}{4}}(I;L^2(\Gamma))}
			\le Ck^{\frac{1}{4}}\Vert e_k\Vert_{I}.
	\end{align*}
	 Similarly, $I_2$ can be bounded as
	\begin{align*}
		\vert I_2\vert
		\le\Vert f\Vert_{L^2(I;L^2(\Omega))}\Vert z-P_kz\Vert_{I}
		\le Ck\Vert f\Vert_{L^2(I;L^2(\Omega))}\Vert z\Vert_{H^1(I;L^2(\Omega))}
		\le Ck\Vert f\Vert_{L^2(I;L^2(\Omega))}\Vert e_k\Vert_{I}.
	\end{align*}
	The term $I_3$ can be estimated by
	\begin{align*}
		\vert I_3\vert&=\Big\vert\sum\limits_{n=1}^M([y_k(u)]_{n-1},(P_kz)^+_{n-1}-z(t_{n-1}))\Big\vert\\
		&\le \Big(\sum\limits_{n=1}^Mk^{-(2s-1)}_n\Vert[y_k(u)]_{n-1}\Vert^2_{L^2(\Omega)}\Big)^{\frac{1}{2}}\Big(\sum\limits_{n=1}^Mk^{2s-1}_n\Vert P^n_kz-z(t_{n-1})\Vert^2_{L^2(\Omega)}\Big)^{\frac{1}{2}}\\
		&\le C\Big(\sum\limits_{n=1}^Mk^{-(2s-1)}_n\Vert[y_k(u)]_{n-1}\Vert^2_{L^2(\Omega)}\Big)^{\frac{1}{2}}\Big(\sum\limits_{n=1}^Mk^{2s}_n\Vert \partial_tz\Vert^2_{L^2(I_n;L^2(\Omega))}\Big)^{\frac{1}{2}}\\
		&\le Ck^s\Big(\sum\limits_{n=1}^Mk^{-(2s-1)}_n\Vert[y_k(u)]_{n-1}\Vert^2_{L^2(\Omega)}\Big)^{\frac{1}{2}}\Vert \partial_tz\Vert_{I}\\
		&\le Ck^s\Big(\sum\limits_{n=1}^Mk^{-(2s-1)}_n\Vert[y_k(u)]_{n-1}\Vert^2_{L^2(\Omega)}\Big)^{\frac{1}{2}}\Vert e_k\Vert_{I},
	\end{align*}
	where we have used the Cauchy-Schwarz inequality and the following estimate:
	\begin{equation}
		\left\Vert P^n_kz-z(t_{n-1})\right\Vert^2_{L^2(\Omega)}\le Ck_n \Vert\partial_tz\Vert_{L^2(I_n;L^2(\Omega))},\quad  n=1,\cdots,M.\nonumber
	\end{equation}
	Applying Lemma \ref{jump_main_esti} in the above inequality, we finish the estimate for $I_3$. Inserting the estimates for $I_1,I_2,I_3$ into \eqref{IE_3}, we finally obtain the desired result of this theorem.
\end{proof}

When $\Omega$ is a convex polytope, we will consider the error estimate for the fully discrete solution of the optimal control problem \eqref{BC_Functioal}. Note that the state variable $y$ and its temporal semi-discretization $y_k(u)$ are less regular, i.e., $y\in L^2(I;H^1(\Omega))\cap H^{\frac{1}{2}}(I;L^2(\Omega))$, $y_k(u)\in L^2(I;H^1(\Omega))$. Therefore, we have to give the stability of $y_k(u)$ under the norm $\Vert\cdot\Vert_{L^2(I;H^1(\Omega))}$.

\begin{Proposition}\label{semi_regularity}
For $u\in L^2(I;H^{1\over 2}(\Gamma))\cap H^{1\over 4}(I;L^2(\Gamma))$, let $y\in L^2(I;H^1(\Omega))\cap H^{\frac{1}{2}}(I;L^2(\Omega))$ and $y_k(u)\in X_k$ be the solution of \eqref{BC_stateE} and \eqref{time_semi_state_e}, respectively. Then we have
	\begin{align*}
		\Vert y_k(u)\Vert_{L^2(I;H^1(\Omega))}&\le C\Big(\Vert y\Vert_{L^2(I;H^1(\Omega))}+\Vert y\Vert_{H^{\frac{1}{2}}(I;L^2(\Omega))}+\Vert u\Vert_{H^{\frac{1}{4}}(I;L^2(\Gamma))}\\
		&+\Vert f\Vert_{L^2(I;L^2(\Omega))}+\Vert \partial_t y\Vert_{L^2(I;H^{-1}(\Omega))}\Big).
	\end{align*}
\end{Proposition}
\begin{proof}
	For each fixed $m=1,\cdots,M$, we can rewrite the semi-discrete scheme (\ref{time_semi_state_e}) into the following time stepping scheme:
	\begin{eqnarray}\label{time_steeping_form}
		\begin{cases}
			&(\nabla y_{k,m}(u),\nabla\varphi)+\big(\frac{1}{k_m}y_{k,m}(u),\varphi\big)=\big(k_m\bar{f}_m+\frac{y_{k,m-1}(u)}{k_m},\varphi\big)\qquad\forall\,\varphi\in H^1_0(\Omega),\\ \nonumber
			&y_{k,0}(u)=y_0\ \hspace{1.4cm}\ \mbox{in}\,\ \Omega,\\
			&y_{k,m}(u)|_{\Gamma}=\tilde{P}^m_ku\hspace{0.7cm}\  \mbox{on}\,\  \Gamma,
		\end{cases}
	\end{eqnarray}
	where
	$\bar{f}_m=\frac{1}{k_m}\displaystyle{\int_{t_{m-1}}^{t_m}f(s)ds}$. Since $\tilde{P}^m_ku\in H^{\frac{1}{2}}(\Gamma)$ and $k_m\bar{f}_m+\frac{y_{k,m-1}(u)}{k_m}\in L^2(\Omega),$ we obtain $y_{k,m}(u)\in H^1(\Omega),\ m=1,\cdots,M$ by using the regularity of elliptic equations (cf. \cite{L. C. EVANS}).
	
	Observing that $P^m_ky-y_{k,m}(u)\in H^1_0(\Omega)$, we use Poincar\'e's inequality to estimate
	\begin{equation}\label{poiu}
		\begin{split}
			\Vert y_k(u)\Vert_{L^2(I;H^1(\Omega))}&\le \Vert P_ky\Vert_{L^2(I;H^1(\Omega))}+\Vert P_ky-y_k(u)\Vert_{L^2(I;H^1(\Omega))}\\
			&\le C\big(\Vert y\Vert_{L^2(I;H^1(\Omega))}+\Vert \nabla(P_ky-y_k(u))\Vert_{L^2(I;L^2(\Omega))}\big).
		\end{split}
	\end{equation}
	Hence, it suffices to estimate the second term of the above inequality. Choosing $\varphi|_{I_m}:=(P_ky-y_k(u))|_{I_m}$ as the test function in the identity (\ref{semi_pro1}) of Lemma \ref{jumping_estimate}, we have
	\begin{equation}
		k_m\Vert \nabla(P_ky-y_k(u))_m\Vert^2_{L^2(\Omega)}+\big([P_ky-y_k(u)]_{m-1},(P_ky-y_k(u))_m\big)=\langle\tilde{g},P_ky-y_k(u)\rangle_{L^2(I_m;H^{-1},H^1_0)}.\nonumber
	\end{equation}
Using the Cauchy-Schwarz inequality we obtain
	\begin{eqnarray*}
		&&k_m\left\Vert \nabla(P_ky-y_k(u))_m\right\Vert^2_{L^2(\Omega)}+\frac{1}{2}\big(\Vert[P_ky-y_k(u)]_{m-1}\Vert^2_{L^2(\Omega)}+\Vert(P_ky-y_k(u))_m\Vert^2_{L^2(\Omega)}\\
		&&-\Vert(P_ky-y_k(u))_{m-1}\Vert^2_{L^2(\Omega)}\big)
		\le\left\Vert\frac{P^m_ky-P^{m-1}_ky}{k_m}\right\Vert^2_{I_m}\left\Vert P_ky-y_k(u)\right\Vert^2_{I_m}\\
		&&+C\Vert\partial_t y\Vert_{L^2(I_m;H^{-1}(\Omega))}\Vert\nabla(P_ky-y_k(u))\Vert_{I_m}.
	\end{eqnarray*}
	Summing up the above inequality for $m=1,\cdots,M$ and applying the Cauchy-Schwarz inequality, we   derive from Theorem \ref{error_semi_d} and Lemma \ref{jump_main_esti} that
	\begin{eqnarray*}
		&&\sum\limits_{m=1}^Mk_m\left\Vert \nabla(P_ky-y_k(u))_m\right\Vert^2_{L^2(\Omega)}+\sum\limits_{m=1}^M\Vert[P_ky-y_k(u)]_{m-1}\Vert^2_{L^2(\Omega)}\\
		&\le& 2\Big(\sum\limits_{m=1}^Mk_m\left\Vert\frac{P^m_ky-P^{m-1}_ky}{k_m}\right\Vert^2_{I_m}\Big)^{\frac{1}{2}}\Big(\sum\limits_{m=1}^Mk^{-1}_m\Vert P_ky-y_k(u)\Vert^2_{I_m}\Big)^{\frac{1}{2}}+C\Vert\partial_t y\Vert^2_{L^2(I;H^{-1}(\Omega))}\\\nonumber
		&\le& Ck^{-\frac{1}{2}}\Big(\sum\limits_{m=1}^M\Vert [P_ky]_{m-1}\Vert^2_{L^2(\Omega)}\Big)^{\frac{1}{2}}\Big(\Vert P_ky-y\Vert_I+\Vert y-y_k(u)\Vert_{I}\Big)+C\Vert\partial_t y\Vert^2_{L^2(I;H^{-1}(\Omega))}\\
		&\le& C \Vert y\Vert_{H^{\frac{1}{2}}(I;L^2(\Omega)}\Big(\Vert y\Vert_{H^{\frac{1}{2}}(I;L^2(\Omega))}+\Vert u\Vert_{H^{\frac{1}{4}}(I;L^2(\partial\Omega))}+\Vert f\Vert_{L^2(I;L^2(\Omega))}\Big)+\Vert \partial_t y\Vert_{L^2(I;H^{-1}(\Omega))}^2\\
		&\le& C\Big(\left\Vert y\right\Vert^2_{H^{\frac{1}{2}}(I;L^2(\Omega))}+\left\Vert u\right\Vert^2_{H^{\frac{1}{4}}(I;L^2(\partial\Omega))}+\left\Vert f\right\Vert^2_{L^2(I;L^2(\Omega))}+\Vert \partial_t y\Vert_{L^2(I;H^{-1}(\Omega))}^2\Big).
	\end{eqnarray*}
	Combining (\ref{poiu})  with the above estimate we finish the proof.
    \end{proof}

	\subsection{Analysis of the spatial discretization error}
	
	When $\Omega$ is a convex polytope, we consider the spatial discretization error estimate for the solution of  (\ref{spatial_dicrete_state}) under the norm $\Vert\cdot\Vert_{L^2(I,L^2(\Omega))}$. For this purpose, we use the Aubin-Nitsche technique to convert the finite element estimate into projection errors. 
	
	We remark that the projection operators $P_k$ and $\tilde{P}_k$ defined respectively in (\ref{pro_function_d}) and (\ref{pro_boundary_condition}) can be extended such that $P_k:L^2(I;H^1(\Omega))\to X_k$ and $\tilde{P}_k:L^2(I;H^{\frac{1}{2}}(\Gamma))\to X_k(\Gamma)$ are well-defined. 
 {The following lemma establishes the relationships between these projections, which will be used in the error estimate of the spatial discretization.
		\begin{Lemma}\label{projection_operator}
			The projections  $P_k$, $\tilde{P}_k$, $P_h$, $\tilde{P}_h$, $P_{kh}$ and $\tilde{P}_{kh}$ in Definition \ref{defn_pprojection} satisfy
			\begin{equation}
				P_{kh}=P_kP_h,\quad \tilde{P}_{kh}=\tilde{P}_k\tilde{P}_h,\quad \gamma_0P_k=\tilde{P}_k\gamma_0,\nonumber
			\end{equation}
			where $\gamma_0:L^2(I;H^1(\Omega))\to L^2(I;H^{\frac{1}{2}}(\Gamma))$ is the trace operator. Particularly, when $ \omega\in L^2(I;H^1(\Omega))$ and $\tilde{\omega}\in L^2(I;H^{\frac{1}{2}}(\Gamma))$, there hold
			\begin{equation}
				P_{kh}\omega=P_hP_k\omega,\quad \tilde{P}_{kh}\tilde{\omega}=\tilde{P}_h\tilde{P}_k\tilde{\omega}.\nonumber
			\end{equation}
		\end{Lemma}
		\begin{proof}
			We first prove $P_{kh}=P_kP_h$. For any given $\omega\in L^2(I;L^2(\Omega))$ and $v_{kh}\in X_{kh}$, there holds
			\begin{align*}
				\int_\Omega\int_0^T(\omega-P_kP_h\omega)v_{kh}dtdx&=\int_\Omega\int_0^T(P_k\omega-P_kP_h\omega)v_{kh}dtdx\\
				&=\sum\limits_{m=1}^M\int_\Omega\int_{I_m}P_k(\omega-P_h\omega)v_{kh}dtdx\\
				&=\sum\limits_{m=1}^M\int_\Omega\int_{I_m}(\omega-P_h\omega)dtv_{kh}dx\\
				&=(\omega-P_h\omega,v_{kh})_{L^2(I;L^2(\Omega))},
			\end{align*}
		    where we have used the orthogonality of $P_k$. Note that $X_{kh}\subseteq L^2(I;V_h)$, if follows from the definition of $P_h$ that the above identity is zero. By the definition of $P_{kh}$ and the uniqueness of $P_{kh}\omega$, we obtain $P_{kh}\omega=P_kP_h\omega$ for any $\omega\in L^2(I;L^2(\omega))$, i.e., $P_{kh}=P_kP_h$. Similarly, we can show that $\tilde{P}_{kh}=\tilde{P}_h\tilde{P}_k$.

			Now, we claim $P_{kh}\omega=P_hP_k\omega$ for any $\omega\in L^2(I;H^1(\Omega))$. For any $v_{kh}\in X_{kh}$, there holds
			\begin{align*}
				\int_\Omega\int_0^T(\omega-P_hP_k\omega)v_{kh}dtdx&=\int_\Omega\int_0^T[(\omega-P_h\omega)+P_h(\omega-P_k\omega)]v_{kh}dtdx\\
				&=\sum\limits_{m=1}^M\int_\Omega\int_{I_m}P_h(\omega-P_k\omega)v_{kh}dtdx\\
				&=\sum\limits_{m=1}^M\int_{I_m}\int_\Omega(\omega-P_k\omega)dxv_{kh}dt\\
				&=(\omega-P_k\omega,v_{kh})_{L^2(I;L^2(\Omega))},
			\end{align*}
			where the orthogonality of $P_h$ is used. 	By the orthogonality of $P_k$ and the inclusion $X_{kh}\subseteq X_k$, the above equality is zero. Using the definition of $P_{kh}$ and the uniqueness of $P_{kh}\omega$, we obtain  $P_{kh}\omega=P_hP_k\omega$ for any $\omega\in L^2(I;H^1(\Omega))$. Similar to the above process, there holds $\tilde{P}_{kh}\tilde{\omega}=\tilde{P}_h\tilde{P}_k\tilde{\omega}$ for any $\omega\in L^2(I;H^{\frac{1}{2}}(\Gamma))$. The relation
			$\gamma_0P_k=\tilde{P}_k\gamma_0$ is obvious. This completes the proof.
		\end{proof}


For any given $\omega\in L^2(I;H^1(\Omega))$, let $P_h\omega$ be the projection of $\omega$ defined in Definition \ref{defn_pprojection}. We split  $P_h\omega$ into
\begin{equation}
	P_h\omega=P^I_h\omega+P^B_h\omega,\quad P^I_h\omega\in L^2(I;V^0_h),\quad P^B_h\omega\in L^2(I;V^\partial_h),\nonumber
\end{equation}
where $L^2(I;V^0_h)$ and $L^2(I;V_h)$ are function spaces  defined in $I$, and valued in $V^0_h$ and $V_h$ respectively, while $V^0_h$ and $V^\partial_h$ are subspaces of $V_h$ spanned by interior node basis functions and boundary node basis functions, respectively.

Then we define a modified projection
\begin{equation}
		\hat{P}_h:L^2(I;H^1(\Omega))\to L^2(I;V_h),\quad
		\omega\mapsto\hat{P}_h\omega\nonumber
\end{equation}
such that $\hat{P}_h\omega:=P^I_h\omega+P^\partial_h\omega$, where $P_h^\partial\omega$ satisfies $P_h^\partial\omega\in L^2(I;V^\partial_h)$ and $P_h^\partial\omega|_{I\times\Gamma}:=\tilde{P}_h\gamma_0\omega$, i.e.,
\begin{equation}
	\big(\omega|_{I\times\Gamma}-P^\partial_h\omega|_{I\times\Gamma},v_h\big)_{L^2(I;L^2(\Gamma))}=0 \qquad \forall v_h\in L^2(I;V_h(\Gamma)).\nonumber
\end{equation}

Note that the above projection is obtained  by replacing the boundary components of the usual $L^2$-projection with  the boundary $L^2$-projection. Therefore, the orthogonality  is only valid on the boundary. In the following proposition, we study the projection error estimation of the temporal DG(0) discretization for parabolic equations.

\begin{Proposition}\label{approximation_prop}
	Let $y_k(u)\in L^2(I;H^1(\Omega))$ be the solution of \eqref{time_semi_state_e}. Then there holds
	\begin{equation}		
	\big\Vert y_k(u)-\hat{P}_hy_k(u)\big\Vert_I+h\big\Vert\nabla(y_k(u)-\hat{P}_hy_k(u))\big\Vert_I\le Ch\big(\Vert y_k(u)\Vert_{L^2(I;H^1(\Omega))}+\Vert u\Vert_{L^2(I;H^{\frac{1}{2}}(\Gamma))}\big).\nonumber
	\end{equation}
\end{Proposition}
\begin{proof}
	By the error estimate of the $L^2$-projection $P_h$, we have
	\begin{eqnarray*}
		&&\Vert y_k(u)-\hat{P}_hy_k(u)\Vert_I+h\Vert\nabla(y_k(u)-\hat{P}_hy_k(u))\Vert_I\\
		&\le& \big(\Vert y_k(u)-P_hy_k(u)\Vert_I+h\Vert\nabla(y_k(u)-P_hy_k(u))\Vert_I\big)\\
		&&+\Vert P_hy_k(u)-\hat{P}_hy_k(u)\Vert_I+h\Vert\nabla(P_hy_k(u)-\hat{P}_hy_k(u))\Vert_I\\
		&\le& Ch\Vert y_k(u)\Vert_{L^2(I;H^1(\Omega))}+\Vert P_hy_k(u)-\hat{P}_hy_k(u)\Vert_I+h\Vert\nabla(P_hy_k(u)-\hat{P}_hy_k(u))\Vert_I.
	\end{eqnarray*}
Then we estimate the latter two terms of the above inequality as follows:
	\begin{eqnarray*}
		&&\Vert P_hy_k(u)-\hat{P}_hy_k(u)\Vert^2_I+h^2\Vert\nabla(P_hy_k(u)-\hat{P}_hy_k(u))\Vert^2_I\\
		&=&\sum\limits_{m=1}^Mk_m\Vert P^m_hy_k(u)-\hat{P}^m_hy_k(u)\Vert^2_{L^2(\Omega)}+h^2\sum\limits_{m=1}^Mk_m\Vert\nabla(P^m_hy_k(u)-\hat{P}^m_hy_k(u))\Vert^2_{L^2(\Omega)}\\
		&=&\sum\limits_{m=1}^Mk_m\sum\limits_{j=1}^{M_h}\big(P^m_hy_k(u)(x_j)-\hat{P}^m_hy_k(u)(x_j)\big)^2\Vert\varphi_j\Vert^2_{L^2(\Omega)}\\
		&&+h^2\sum\limits_{m=1}^M\!k_m\!\!\sum\limits_{j=1}^{M_h}\big(P^m_hy_k(u)(x_j)-\hat{P}^m_hy_k(u)(x_j)\big)^2\Vert\nabla\varphi_j\Vert^2_{L^2(\Omega)}\\
		&\leq& Ch^N\sum\limits_{m=1}^Mk_m\sum\limits_{j=1}^{M_h}\big(P^m_hy_k(u)(x_j)-\hat{P}^m_hy_k(u)(x_j)\big)^2\\
		&\leq& Ch\sum\limits_{m=1}^Mk_m\Vert P^m_hy_k(u)-\hat{P}^m_hy_k(u)\Vert^2_{L^2(\Gamma)}\\
		&=&Ch\Vert P_hy_k(u)-\hat{P}_hy_k(u)\Vert^2_{L^2(I;L^2(\Gamma))},
	\end{eqnarray*}
	where $\hat{P}^m_hy_k(u):=\hat{P}_hy_k(u)|_{I_m},\,m=1,\cdots,M$, and $\{\varphi_j\}_{j=1}^{M_h}$ is a family of basis functions corresponding to the boundary nodes $\{x_j\}_{j=1}^{M_h}$. Subsequently,
	\begin{eqnarray*}
		&&Ch^{\frac{1}{2}}\Vert P_hy_k(u)-\hat{P}_hy_k(u)\Vert_{L^2(I;L^2(\Gamma))}\\
		&\le& Ch^{\frac{1}{2}}\Vert P_hy_k(u)-y_k(u)\Vert_{L^2(I;L^2(\Gamma))}+Ch^{\frac{1}{2}}\Vert y_k(u)-\hat{P}_hy_k(u)\Vert_{L^2(I;L^2(\Gamma))}\\
		&\le& C\Big(\Vert y_k(u)-P_hy_k(u)\Vert_{L^2(I;L^2(\Omega))}+h\Vert y_k(u)-P_hy_k(u)\Vert_{L^2(I;H^1(\Omega))}\\
		&&+h^{\frac{1}{2}}\Vert \tilde{P}_ku-\tilde{P}_h\tilde{P}_ku\Vert_{L^2(I;L^2(\Gamma))}\Big)\\
		&\le& Ch\Vert y_k(u)\Vert_{L^2(I;H^1(\Omega))}+Ch^{\frac{1}{2}}\Vert \tilde{P}_k(u-\tilde{P}_hu)\Vert_{L^2(I;L^2(\Gamma))}\\
		&\le& Ch\Vert y_k(u)\Vert_{L^2(I;H^1(\Omega))}+Ch^{\frac{1}{2}}\Vert u-\tilde{P}_hu\Vert_{L^2(I;L^2(\Gamma))}\\
		&\le&  Ch\big(\Vert y_k(u)\Vert_{L^2(I;H^1(\Omega))}+\Vert u\Vert_{L^2(I;H^{\frac{1}{2}}(\Gamma))}\big),
	\end{eqnarray*}
	where we have used Lemma \ref{projection_operator}, the trace theorem, the stability and approximation properties of projection operators. This completes the proof.
   \end{proof}
	
	The fully discrete scheme (\ref{spatial_dicrete_state}) for parabolic equations is called the DG(0)-CG(1) scheme, which can be viewed as the spatial discretization of (\ref{time_semi_state_e}). 
Since $X^0_{kh}\subseteq X^0_k$, there holds the following orthogonality relation:
		\begin{eqnarray}
		B(y_k(u)-y_{kh}(u),\varphi_{kh})=B(y-y_{kh}(u),\varphi_{kh})=0 \qquad\forall\varphi_{kh}\in X^0_{kh}.\nonumber
		\end{eqnarray}

	\begin{Theorem}\label{error_space_esti}
		Let $y_k(u)\in L^2(I;H^1(\Omega))$ and $y_{kh}(u)\in X_{kh}$ be the solution of \eqref{time_semi_state_e} and \eqref{spatial_dicrete_state}, respectively. Then we have 
		\begin{equation}
			\Vert y_k(u)-y_{kh}(u)\Vert_{L^2(I;L^2(\Omega))}\le Ch\big(\Vert y_k(u)\Vert_{L^2(I;H^1(\Omega))}+\Vert u\Vert_{L^2(I;H^{\frac{1}{2}}(\Gamma))}\big).\nonumber
		\end{equation}
	\end{Theorem}
	\begin{proof}
	Let $e_h:=y_k(u)-y_{kh}(u)=\eta_h+\zeta_h$, $\eta_h:=y_k(u)-\hat{P}_hy_k(u)$, $\zeta_h:=\hat{P}_hy_k(u)-y_{kh}(u)$. Then 
		\begin{equation}\label{Ine_1}
			\Vert e_h\Vert^2_{L^2(I;L^2(\Omega))}=(e_h,\eta_h)_{L^2(I;L^2(\Omega))}+(e_h,\zeta_h)_{L^2(I;L^2(\Omega))}.
		\end{equation}
		By using the Cauchy-Schwarz inequality, it suffices to estimate the second term. To this end, we use the Aubin-Nitsche technique and let $z$ be the solution of (\ref{T_back_eq}) with $g:=e_h$.
		
		Let $z_k\in X^0_k$ and $z_{kh}\in X^0_{kh}$ be the temporal semi-discretization and fully discretization of the  equation (\ref{T_back_eq}). By Lemma \ref{projection_operator}, there holds $\zeta_h\in X^0_{kh}$. Note that $X^0_{kh}\subseteq X^0_k$, then
				\begin{eqnarray*}
			&&-(e_h,\zeta_h)_{L^2(I;L^2(\Omega))}=-B(\zeta_h,z_{kh})
			=B(\eta_h,z_{kh})\\
			&=&(\nabla\eta_h,\nabla z_{kh})_I-\sum_{m=1}^{M-1}(\eta^-_{h,m},[z_{kh}]_m)+(\eta^-_{h,M},z^-_{kh,M})\\
			&=&(\nabla\eta_h,\nabla z_k)_I-(\nabla\eta_h,\nabla (z_k-z_{kh}))_I-\sum_{m=1}^M(\eta^-_{h,m},[z_{kh}]_m)\\
			&=&-(\eta_h,\Delta z_k)_I+(\eta_h,\partial_n z_k)_{L^2(I;L^2(\Gamma))}-(\nabla\eta_h,\nabla (z_k-z_{kh}))_I-\sum_{m=1}^M(\eta^-_{h,m},[z_{kh}]_m)\\
			&=&-(\eta_h,\Delta z_k)_I+(u-\tilde{P}_hu,\partial_n z_k-\tilde{P}_h\partial_n z_k)_{L^2(I;L^2(\Gamma))}-(\nabla\eta_h,\nabla (z_k-z_{kh}))_I-\sum_{m=1}^M(\eta^-_{h,m},[z_{kh}]_m),
		\end{eqnarray*}
		where $[z_{kh}]_M=-z^-_{kh,M}$, and we have used $(\eta_h,\partial_n z_k)_{L^2(I;L^2(\Gamma))}=(u-\tilde{P}_hu,\partial_n z_k-\tilde{P}_h\partial_n z_k)_{L^2(I;L^2(\Gamma))}$. In fact, it follows from Lemma \ref{projection_operator} and the orthogonality of the projection $\tilde{P}_h$ that
		\begin{align*}
			(\eta_h,\partial_n z_k)_{L^2(I;L^2(\Gamma))}&=(y_k(u)-\tilde{P}_hy_k(u),\partial_n z_k)_{L^2(I;L^2(\Gamma))}\\
			&=(\tilde{P}_ku-\tilde{P}_h\tilde{P}_ku,\partial_n z_k)_{L^2(I;L^2(\Gamma))}\\
			&=(\tilde{P}_k(u-\tilde{P}_hu),\partial_n z_k)_{L^2(I;L^2(\Gamma))}\\
			&=(u-\tilde{P}_hu,\partial_n z_k-\tilde{P}_h\partial_n z_k)_{L^2(I;L^2(\Gamma))}.
		\end{align*}
		Hence,
		\begin{align}\label{Ine_3}\notag
			\vert (e_h,\zeta_h)_{L^2(I;L^2(\Omega))}\vert&\le\Vert\eta_h \Vert_I\Vert\Delta z_k\Vert_I+\Vert u-\tilde{P}_hu\Vert_{L^2(I;L^2(\Gamma))}\Vert(I-\tilde{P}_h)\partial_n z_k\Vert_{L^2(I;L^2(\Gamma))}\\\notag
			&+\left\Vert\nabla\eta_h\Vert_I\right\Vert\nabla (z_k-z_{kh})\Vert_I+\Big(\sum_{m=1}^Mk_m\Vert\eta_{h,m}\Vert^2_{L^2(\Omega)}\Big)^{\frac{1}{2}}\Big(\sum_{m=1}^Mk^{-1}_m\Vert[z_{kh}]_m\Vert^2_{L^2(\Omega)}\Big)^{\frac{1}{2}}\\
			&\le C\Big(\Vert\eta_h \Vert_I+h^{\frac{1}{2}}\Vert u-\tilde{P}_hu\Vert_{L^2(I;L^2(\Gamma))}+h\Vert\nabla\eta_h\Vert_I+\Vert\eta_h\Vert_I\Big)\Vert e_h\Vert_I.\notag
		\end{align}
	 Here  we have used the following three estimates:
		\begin{align*}
			\Vert\nabla (z_k-z_{kh})\Vert_I&\le Ch\Vert e_h\Vert_I,\\
			\big\Vert(I-\tilde{P}_h)\partial_n z_k\big\Vert_{L^2(I;L^2(\Gamma))}&\le Ch^{\frac{1}{2}}\Vert e_h\Vert_I,\\
			\Vert\Delta z_k\Vert_I+\Vert z_k\Vert_I+\Big(\sum_{m=1}^Mk^{-1}_m\Vert[z_{kh}]_m\Vert^2_{L^2(\Omega)}\Big)^{\frac{1}{2}}&\le C\Vert e_h \Vert_I,
		\end{align*}
	where the third one can be found in \cite[Corollary 4.2 and 4.7]{MeidnerVexler}. Therefore, we only verify the first two estimates. The first estimate can be derived by using the inverse estimate and the projection error estimate:
		\begin{align*}
			\Vert\nabla (z_k-z_{kh})\Vert_I&\le\Vert\nabla (z_k-P_hz_k)\Vert_I+Ch^{-1}\Vert P_hz_k-z_{kh}\Vert_I\\
			&\le Ch\Vert\nabla^2 z_k\Vert_I+Ch^{-1}(\Vert P_hz_k-z_k\Vert_I+\Vert z_k-z_{kh}\Vert_I)\\
			&\le Ch\Vert z_k\Vert_{L^2(I;H^2(\Omega))}\\
			&\le Ch(\Vert z_k\Vert_I\!+\Vert\Delta z_k\Vert_I)\\
			&\le Ch\Vert e_h\Vert_I,
		\end{align*}
	where the error estimate and the stability of the temporal semi-discretization (cf. \cite[Theorem 5.1 and Corollary 4.2]{MeidnerVexler}) have been applied. Again, the second estimate can be verified by the stability of the temporal semi-discretization (cf. \cite[Corollary 4.2]{MeidnerVexler})
		\begin{align*}
			\big\Vert(I-\tilde{P}_h)\partial_n z_k\big\Vert_{L^2(I;L^2(\Gamma))}&\le Ch^{\frac{1}{2}}\Vert\partial_n z_k\Vert_{L^2(I;H^{\frac{1}{2}}(\Gamma))}
			\le Ch^{\frac{1}{2}}\Vert z_k\Vert_{L^2(I;H^2(\Omega))}
			\le Ch^{\frac{1}{2}}\Vert e_h\Vert_I.
		\end{align*}
		Finally, combining (\ref{Ine_1}), Proposition \ref{approximation_prop}, the Cauchy-Schwarz and Young's inequalities, and the projection error estimate  completes the proof.
	\end{proof}
\begin{Remark}
Combining Theorems \ref{error_semi_d} and \ref{error_space_esti} we obtain a priori error estimate $O(h+k^{1\over 2})$ for  the fully discretization of parabolic equations with inhomogeneous Dirichlet data posed on polytopes, which improves the result in \cite{DAFrenchandJTKing1993} by removing the mesh size condition $k=O(h^2)$. 
\end{Remark}	
	
	\section{Error estimates for the optimal control problem}
		\setcounter{equation}{0}

	    In this section, we give the main results of this article, namely, the error estimates between the solutions of the continuous optimal control problem and discrete optimal control problems. The first subsection gives the optimal order of convergence for the temporal semi-discrete optimal control problem \eqref{time_semi_control} in the case that $\Omega$ is smooth or polytopal, and the second subsection is devoted to the error estimate between the continuous problem \eqref{BC_Functioal} and fully discrete optimal control problem \eqref{space_semi_control} in the case that $\Omega$ is polytopal.
	
\subsection{Estimates for the semi-discrete adjoint state}

	\begin{Lemma}\label{QQ1}
Let $\bar{z}\in L^2(I;H^2(\Omega)\cap H^1_0(\Omega))\cap H^1(I;L^2(\Omega))$ be the adjoint state of the optimal control problem \eqref{BC_Functioal} and  $\hat{z}_k$ be its semi-discrete finite element solution, i.e., $\hat{z}_k\in X_k^0$ satisfies
\begin{equation}\label{EE10}
B(\varphi_k,\hat{z}_k)=(\bar{y}-y_d,\varphi_k)_I\qquad \forall\varphi_k\in X^0_k.
\end{equation}%
Let $\pi^r_k\bar{z}\in X_k$ be the pointwise in time interpolation of $\bar{z}$  defined by \eqref{def_interpolation}. Then we have 
\begin{equation}\label{EE11}
\Vert \nabla(\bar{z}-\hat{z}_k)\Vert_{I}\le C\left\Vert \nabla\left(\bar{z}-\pi^r_k\bar{z}\right)\right\Vert_{I}.
\end{equation}
    \end{Lemma}
\begin{proof}
Observing that
\begin{equation}\label{ESS}
\Vert \nabla(\bar{z}-\hat{z}_k)\Vert_{I}\le \Vert \nabla(\bar{z}-\pi^r_k\bar{z})\Vert_{I}+\Vert \nabla(\pi^r_k\bar{z}-\hat{z}_k)\Vert_{I},
\end{equation}
we only need to estimate the second term on the right-hand side of the above inequality.

{Adding the expression \eqref{bilinearF}  and the dual expression \eqref{dbilinearF} of the bilinear form $B$ for $v=w=\pi^r_k\bar{z}-\hat{z}_k\in X_k$, we obtain
\begin{equation}
\begin{split}
 \Vert \nabla(\pi^r_k\bar{z}-\hat{z}_k)\Vert_{I}^2&\le B(\pi^r_k\bar{z}-\hat{z}_k,\pi^r_k\bar{z}-\hat{z}_k)\\
 &=B(\pi^r_k\bar{z}-\hat{z}_k,\pi^r_k\bar{z}-\bar{z})\\
 &=(\nabla(\pi^r_k\bar{z}-\bar{z}),\nabla(\pi^r_k\bar{z}-\hat{z}_k ))_I-\sum_{m=1}^M((\pi^r_k\bar{z}-\bar{z})^-_m,[\pi^r_k\bar{z}-\hat{z}_k]_m)\\
 &\le (\nabla(\pi^r_k\bar{z}-\bar{z}),\nabla(\pi^r_k\bar{z}-\hat{z}_k ))_I\\
 &\le\Vert \nabla(\bar{z}-\pi^r_k\bar{z})\Vert_{I}\Vert \nabla(\hat{z}_k-\pi_k^r\bar{z})\Vert_{I},
\end{split}
\end{equation}
where we have used $\bar{z}(t_m^-)=\pi^r_k\bar{z}|_{I_m}=\bar{z}(t_m)$}. Then, combining the above estimate with \eqref{ESS} gives the desired estimate \eqref{EE11}. This completes the proof.
\end{proof}

\begin{Lemma}\label{Gradient_jump_E}
Assume that $\Omega$ is a bounded smooth domain and the compatibility condition \eqref{compatibility_condition} holds for the adjoint equation \eqref{Conjugate_EQ} with $g:=\bar{y}-y_d$, where $\bar{y}$ is the optimal state and $y_d\in L^2(I;H^1(\Omega))\cap H^{\frac{1}{2}}(I;L^2(\Omega))$. Let $\hat{z}_k\in X_k$ be the finite element approximation of $\bar{z}$ defined by \eqref{EE10}, then $\hat{z}_k\in L^2(I;H^3(\Omega)\cap H_0^1(\Omega))$ and there holds
\begin{equation}\label{jump_adjoint}
\sum_{m=1}^Mk_m^{-1}\left\Vert\nabla[\hat{z}_k]_m\right\Vert^2_{L^2(\Omega)}\le C\Vert\bar{z}\Vert^2_{H^1(I;H^1(\Omega))},
\end{equation}
where $[\hat{z}_k]_M:=-\hat{z}_{k,M-1}=-\hat{z}_{k}|_{I_{M}}$. 
\end{Lemma}
 \begin{proof}
For each fixed $m=1,\cdots,M$, we denote by $\hat{z}_{k,m-1}:=\hat{z}_k|_{I_m}$ and rewrite the equation \eqref{EE10} as the following time-steeping scheme: For any $\varphi_k\in P_0(I_m,H^1_0(\Omega))$ there holds
   	\begin{equation}\label{semi_discrete_adjoint_EQ}
     		\begin{cases}
     			-\left(\frac{\hat{z}_{k,m}-\hat{z}_{k,m-1}}{k_m},\varphi_k\right)_{I_m}+\left(\nabla \hat{z}_{k,m-1},\nabla\varphi_k\right)_{I_m}=(P_k(\bar{y}-y_d),\varphi_k)_{I_m},\\
                \hat{z}_{k,M}=0.
     		\end{cases}
     	\end{equation}
Since $\frac{\hat{z}_{k,m}}{k_m}+P_k(\bar{y}-y_d)\in H^1(\Omega)$, there holds $\hat{z}_{k,m-1}\in H^3(\Omega)\cap H_0^1(\Omega)$ by the regularity result of elliptic equations (cf. \cite[Theorem 1.8, Chapter I]{V. Girault}).

On the other hand, the adjoint equation \eqref{Conjugate_EQ} satisfies the following weak form on $I_m$:
    	\begin{equation}\label{weak_form_adjoint}
     		\begin{cases}
     			-\left(\frac{[\pi_k^l \bar{z}]_m }{k_m},\varphi_k\right)_{I_m}+\left(\nabla\bar{z},\nabla\varphi_k\right)_{I_m}=(P_k(\bar{y}-y_d),\varphi_k)_{I_m}\quad\forall\varphi_k\in P_0(I_m,H^1_0(\Omega)),\\
                \bar{z}(t_M)=0,
     		\end{cases}
     	\end{equation}
 where $[\pi_k^l \bar{z}]_m=\bar{z}(t_m)-\bar{z}(t_{m-1})$. Subtracting \eqref{weak_form_adjoint} from \eqref{semi_discrete_adjoint_EQ}, we obtain
 \begin{equation}
 -\left(\frac{[\pi_k^l\bar{z}-\hat{z}_k]_m}{k_m},\varphi_k\right)_{I_m}+\left(\nabla(\bar{z}-\hat{z}_{k,m-1}),\nabla\varphi_k\right)_{I_m}=0\quad\forall\varphi_k\in P_0(I_m,H^1_0(\Omega)),\nonumber
 \end{equation}
 i.e.,
 \begin{equation}\nonumber
 \begin{split}
 -\left(\frac{[\pi_k^l\bar{z}-\hat{z}_k]_m}{k_m},\varphi_k\right)_{I_m}+\left(\nabla(\pi_k^l\bar{z}-\hat{z}_{k,m-1}),\nabla\varphi_k\right)_{I_m}=(\nabla(\pi_k^l\bar{z}-\bar{z}),\nabla\varphi_k)_{I_m}.
  \end{split}
 \end{equation}
Taking $\varphi_k=[\pi_k^l\bar{z}-\hat{z}_k]_m$ in the above equality, we obtain
 \begin{equation}\label{identity_1}
 \begin{split}
(\nabla(\pi_k^l\bar{z}-\hat{z}_{k,m-1}),\nabla[\pi_k^l\bar{z}-\hat{z}_k]_m)_{I_m}-\frac{\left\Vert[\pi_k^l\bar{z}-\hat{z}_k]_m\right\Vert^2_{I_m}}{k_m}=(\nabla(\pi_k^l\bar{z}-\bar{z}),\nabla[\pi_k^l\bar{z}-\hat{z}_k]_m)_{I_m}.
  \end{split}
 \end{equation}
 Note that
 \begin{align*}
 &\left(\nabla(\pi_k^l\bar{z}-\hat{z}_{k,m-1}),\nabla[\pi_k^l\bar{z}-\hat{z}_k]_m\right)_{I_m}\\
 =&-\frac{1}{2}\Vert\nabla[\pi_k^l\bar{z}-\hat{z}_k]_m\Vert^2_{I_m}+\frac{1}{2}\Big(\Vert\nabla(\pi_k^l\bar{z}-\hat{z}_k)^+_m\Vert^2_{I_m}-\Vert\nabla(\pi_k^l\bar{z}-\hat{z}_k)^-_{m}\Vert^2_{I_m}\Big),\nonumber
 \end{align*}
 where $(\pi_k^l\bar{z}-\hat{z}_k)^+_M:=0$. Therefore, the equality \eqref{identity_1} is equivalent to
 \begin{equation}
 \begin{split}
 &-\frac{\Vert[\pi_k^l\bar{z}-\hat{z}_k]_m\Vert^2_{I_m}}{k_m}-\frac{1}{2}\left(\Vert\nabla[\pi_k^l\bar{z}-\hat{z}_k]_m\Vert^2_{I_m}-\Vert\nabla(\pi_k^l\bar{z}-\hat{z}_k)^+_m\Vert^2_{I_m}+\Vert\nabla(\pi_k^l\bar{z}-\hat{z}_k)^-_{m}\Vert^2_{I_m}\right)\\\nonumber
 &=(\nabla(\pi_k^l\bar{z}-\bar{z}),\nabla[\pi_k^l\bar{z}-\hat{z}_k]_m)_{I_m}.
  \end{split}
 \end{equation}
 Thus, we obtain the following estimate:
 \begin{equation}
 \begin{split}
 &\frac{\left\Vert[\pi_k^l\bar{z}-\hat{z}_k]_m\right\Vert^2_{I_m}}{k_m}+\frac{1}{2}\Vert\nabla[\pi_k^l\bar{z}-\hat{z}_k]_m\Vert^2_{I_m}-\Vert\nabla((\pi_k^l\bar{z})(t_m)-\hat{z}_{k,m})\Vert^2_{I_m}\\
 &+\Vert\nabla((\pi_k^l\bar{z})(t_{m-1})-\hat{z}_{k,m-1})\Vert^2_{I_m}
 \le C\Vert\nabla(\pi_k^l\bar{z}-\bar{z})\Vert^2_{I_m}.
  \end{split}
 \end{equation}

 Using the condition \eqref{condition1} we are led to
 \begin{equation}\nonumber
 \begin{split}
 \frac{\left\Vert[\pi_k^l\bar{z}\!-\!\hat{z}_k]_m\right\Vert^2_{I_m}}{k_m}\!+\!\frac{1}{2}\left\Vert\nabla[\pi_k^l\bar{z}-\hat{z}_k]_m\right\Vert^2_{I_m}\!
 \le&C\left\Vert\nabla(\pi_k^l\bar{z}-\bar{z})\right\Vert^2_{I_m}+\Vert\nabla((\pi_k^l\bar{z})(t_m)-\hat{z}_{k,m})\Vert^2_{I_m}\\
 \le& C\Vert\nabla(\pi_k^l\bar{z}-\bar{z})\Vert^2_{I_m}+\frac{k_m}{k_{m+1}}\Vert\nabla(\pi_k^l\bar{z}-\hat{z}_{k})\Vert^2_{I_{m+1}}\\
 \le& C(\Vert\nabla(\pi_k^l\bar{z}-\bar{z})\Vert^2_{I_m}\!+\!\Vert\nabla(\pi_k^l\bar{z}-\bar{z})\Vert^2_{I_{m+1}}+\Vert\nabla(\bar{z}-\hat{z}_{k})\Vert^2_{I_{m+1}}).
  \end{split}
 \end{equation}
Summing up  for $m=1,2,\cdots,M$ on both sides of the above estimate, we obtain
  \begin{equation}\label{Estimte1}
 \begin{split}
 \sum_{m=1}^M\left\Vert[\pi_k^l\bar{z}-\hat{z}_k]_m\right\Vert^2_{L^2(\Omega)}+k_m\Vert\nabla[\pi_k^l\bar{z}-\hat{z}_k]_m\Vert^2_{L^2(\Omega)} &\le C(\Vert\nabla(\pi_k^l\bar{z}-\bar{z})\Vert^2_{I}+\Vert\nabla(\bar{z}-\hat{z}_{k})\Vert^2_{I})\\
&\le Ck^2 \Vert\bar{z}\Vert^2_{H^1(I;H^1(\Omega))},
  \end{split}
 \end{equation}
 where Lemma \ref{QQ1} has been used. Besides, the term $\sum_{m=1}^M\Vert\nabla[\pi_k^l\bar{z}]_m\Vert^2_{L^2(\Omega)}$ can be estimated as
   \begin{equation}\label{Estimte2}
 \begin{split}
 \sum_{m=1}^Mk_m^{-1}\Vert\nabla[\pi_k^l\bar{z}]_m\Vert^2_{L^2(\Omega)}&=\sum_{m=1}^Mk_m^{-1}\Vert\nabla\bar{z}(t_m)-\nabla\bar{z}(t_{m-1})\Vert^2_{L^2(\Omega)}\\
 &=\sum_{m=1}^Mk_m^{-1}\Big\Vert\nabla\int^{t_m}_{t_{m-1}}\partial_t \bar{z}dt\Big\Vert^2_{L^2(\Omega)}\\
 &\le \sum_{m=1}^M\int^{t_m}_{t_{m-1}}\Vert\nabla\partial_t \bar{z}\Vert^2_{L^2(\Omega)}dt\\
 &\le C\left\Vert\bar{z}\right\Vert_{H^1(I;H^1(\Omega))}.
  \end{split}
 \end{equation}
Finally,  using \eqref{Estimte1}, \eqref{Estimte2} and \eqref{condition1} we get the desired result   \eqref{jump_adjoint}.
\end{proof}

\begin{Lemma}\label{Eer_S}
Let $\Omega$ be a bounded smooth domain and $\hat{z}_k\in X_k$ be the solution of \eqref{EE10}. For any  $y_d\in L^2(I;H^{2s-\frac{1}{2}}(\Omega))\cap H^{s-\frac{1}{4}}(I;L^2(\Omega))$ ($s\in[\frac{1}{4},\frac{3}{4}]$), if the optimal state $\bar{y}\in L^2(I;H^{2s-\frac{1}{2}}(\Omega))\cap H^{s-\frac{1}{4}}(I;L^2(\Omega))$, the optimal adjoint state $\bar{z}\in L^2(I;H^{2s+\frac{3}{2}}(\Omega)\cap H^1_0(\Omega))\cap H^{s+\frac{3}{4}}(I;L^2(\Omega))$ and $\partial_t\bar{z}\in L^2(I;H^{2(s-\frac{1}{4})}(\Omega))\cap H^{s-\frac{1}{4}}(I;L^2(\Omega))$, then we have
\begin{equation}\label{ST_adjoint_state_estimate}
\begin{split}
&\Vert\bar{z}-\hat{z}_k\Vert_{L^2(I;H^{2(s-\frac{1}{4})}(\Omega))}\le Ck\left(\Vert \bar{y}-y_d\Vert_{L^2(I;H^{2s-\frac{1}{2}}(\Omega))}+\Vert \bar{y}-y_d\Vert_{H^{s-\frac{1}{4}}(I;L^2(\Omega))}\right),\\
&\Vert \hat{z}_k\Vert_{L^2(I;H^{2s+\frac{3}{2}}(\Omega))}\le C\left(\Vert \bar{y}-y_d\Vert_{L^2(I;H^{2s-\frac{1}{2}}(\Omega))}+\Vert \bar{y}-y_d\Vert_{H^{s-\frac{1}{4}}(I;L^2(\Omega))}\right).
\end{split}
\end{equation}
\end{Lemma}
\begin{proof}
First, we prove the finite element error between $\bar{z}$ and $\hat{z}_k$. The proof is divided into three steps.

{\it Step 1}. We consider the case $s=\frac{1}{4}$, i.e.,
\begin{align}
\Vert\bar{z}-\hat{z}_k\Vert_{L^2(I;L^2(\Omega))}\le Ck \Vert\bar{z}\Vert_{H^1(I;L^2(\Omega))},
\end{align}
which is a direct consequence of \cite[Theorem 5.1]{MeidnerVexler}.

{\it Step 2}. We consider the case $s=\frac{3}{4}$. By Poincar\'e's inequality and Lemma \ref{QQ1}, we obtain
\begin{equation}
\Vert\bar{z}-\hat{z}_k\Vert_{L^2(I;H^1(\Omega))}
\le C\left\Vert \nabla(\bar{z}-\pi^r_k\bar{z})\right\Vert_{I}
\le Ck\Vert\bar{z}\Vert_{H^1(I;H^1(\Omega))}.
\end{equation}

{\it Step 3}. The case $\frac{1}{4}<s<\frac{3}{4}$ follows from interpolations. To do this, denote by $\mathcal{G}$  a linear operator from $\bar{z}$ to $\bar{z}-\hat{z}_k$. The above two estimates imply that $\Vert \mathcal{G}\Vert_{H^1(I;L^2(\Omega))\to L^2(I;L^2(\Omega))}\le Ck$ and $\Vert \mathcal{G}\Vert_{H^1(I;H^1(\Omega))\to L^2(I;H^1(\Omega))}\le Ck$, so that the desired error estimate follows from Proposition 14.1.5 and Theorem 14.2.3 in \cite{BrennerScott2008}.

Next, we prove the stability estimate in \eqref{ST_adjoint_state_estimate}. Similar to the above analysis, the estimate is proved by interpolation estimates and the main procedure includes three steps. Setting $\hat{z}_{k,m-1}:=\hat{z}_k|_{I_m}\in H^2(\Omega)\cap H_0^1(\Omega)$, the equation \eqref{EE10} can be rewritten as the time-stepping scheme:
   	\begin{equation}
     		\begin{cases}
     			-\frac{\hat{z}_{k,m}-\hat{z}_{k,m-1}}{k_m}-\Delta \hat{z}_{k,m-1}=P_k(\bar{y}-y_d)\qquad &\mbox{in}\quad\Omega,\,\, m=1,\cdots,M,\\ \nonumber
     			\hat{z}_{k,m}=0\quad&\mbox{on}\ \ \,\Gamma,\,\,m=1,\cdots,M,\\
                \hat{z}_{k,M}=0&\mbox{in}\quad \,\Omega,\\
     		\end{cases}
     	\end{equation}
 which is a family of elliptic equations on the interval $I_m$.

 We first consider the case $s =\frac{3}{4}$. By the regularity of elliptic equations, there holds $\hat{z}_{k,m-1}\in H^3(\Omega)\cap H^1_0(\Omega)$. Moreover, we can obtain that 
\begin{align*}
 \Vert\hat{z}_{k,m-1}\Vert^2_{H^3(\Omega)}&\le C \Big(\Vert P_k(\bar{y}-y_d)\Vert^2_{H^1(\Omega)}+k_m^{-2}\Vert[\hat{z}_k]_m\Vert^2_{H^1(\Omega)}\Big)\\
 &\le C \Big(\Vert P_k(\bar{y}-y_d)\Vert^2_{H^1(\Omega)}+k_m^{-2}\Vert\nabla[\hat{z}_k]_m\Vert^2_{L^2(\Omega)}\Big),
 \end{align*}
 where $[\hat{z}_k]_M:=-\hat{z}_{k,M-1}$. Integration in $I_m$ on both sides of the above inequality and summing up for $m$, we can obtain
 \begin{align*}
 \Vert\hat{z}_{k}\Vert^2_{L^2(I;H^3(\Omega))}&\le C\Big(\Vert \bar{y}-y_d\Vert^2_{L^2(I;H^1(\Omega))}+\sum^M_{m=1}k^{-1}_m\Vert\nabla[\hat{z}_{k}]_m\Vert^2_{L^2(\Omega)}\Big)\\
 &\le C\left(\Vert\bar{y}-y_d\Vert^2_{L^2(I;H^1(\Omega))}+\Vert\bar{z}\Vert^2_{H^1(I;H^1(\Omega))}\right)\\
 &\le C\Big(\Vert\bar{y}-y_d\Vert^2_{L^2(I;H^1(\Omega))}+\Vert\bar{y}-y_d\Vert^2_{H^{\frac{1}{2}}(I;L^2(\Omega))}\Big),
 \end{align*}
where we have used Lemma \ref{Gradient_jump_E} and Lemma \ref{back_regularity}.

The cases $s=\frac{1}{4}$ can be shown by a similar argument. In fact, we can derive the following stability estimate (cf. \cite[Corollary 4.2]{MeidnerVexler}):
\begin{equation}
\Vert\hat{z}_{k}\Vert^2_{L^2(I;H^2(\Omega))}\le C\Vert\bar{y}-y_d\Vert^2_{L^2(I;L^2(\Omega))}.
\end{equation}

The case
$\frac{1}{4}<s<\frac{3}{4}$ now follows from interpolations, where the following interpolation spaces will be used in this process:
\begin{align*}
&[L^2(I;H^3(\Omega)),L^2(I;H^2(\Omega))]_{\frac{3}{2}-2s}=L^2(I;H^{2s+\frac{3}{2}}(\Omega)),\\
&[L^2(I;H^1(\Omega))\cap H^{\frac{1}{2}}(I;L^2(\Omega)),L^2(I;L^2(\Omega))]_{\frac{3}{2}-2s}=L^2(I;H^{2s-\frac{1}{2}}(\Omega))\cap H^{s-\frac{1}{4}}(I;L^2(\Omega)),
\end{align*}
 which are classical and can be found in many literatures, e.g., \cite{J.L1}. This completes the proof.
\end{proof}
\begin{Remark}
Using the regularity results in Theorem \ref{optimall_regularity} and Theorem \ref{regularity_cont_cnvex}, we observe that Lemma \ref{Eer_S} holds for $s\in [\frac{1}{4},\frac{3}{4}]$ when $\Omega$ is a smooth domain, and for $s=\frac{1}{4}$ when $\Omega$ is a polytope. Based on this lemma, we are able to derive optimal orders of convergence for the solution of temporal semi-discrete control problems in next subsection.
\end{Remark}
		
	\subsection{Error estimates for the temporal semi-discretization}
     In order to estimate the error for the temporal semi-discrete solution of \eqref{time_semi_control}, we establish the stability of the semi-discrete solution of the parabolic inhomogeneous boundary value problem with respect to the Dirichlet data in the following lemma.
     \begin{Lemma}\label{semi_stable_estimate}
     	For any given $u_k\in X_k(\Gamma)$, let $y_k\in X_k$ be the solution of the following equation:
     	\begin{equation}\label{EQF_1}
     			B(y_k,\varphi_k)=0\qquad\forall\varphi_k\in X_k^0,\quad
     			y_k|_{I\times\Gamma}=u_k.
     	\end{equation}
 Then there holds
     	\begin{equation}
     		\Vert y_k\Vert_I\le C\Vert u_k\Vert_{ L^2(I;L^2(\Gamma))}.\nonumber
     	\end{equation}
     \end{Lemma}
     \begin{proof}
     	Let $z_k\in X_k^0$ satisfy the equation
     	\begin{equation}
     		B(\varphi_k,z_k)=(y_k,\varphi_k)_I\qquad\forall\varphi_k\in  X_k^0.\nonumber
     	\end{equation}
     	For each fixed $m=1,\cdots,M$, setting $z_{k,m-1}:=z_k|_{I_m}\in H^2(\Omega)\cap H_0^1(\Omega)$, then the above equation can be rewritten as the following time-stepping scheme:
     	\begin{equation}
     		\begin{cases}
     			-\frac{z_{k,m}-z_{k,m-1}}{k_m}-\Delta z_{k,m-1}=y_{k,m}\qquad &\mbox{in}\quad\Omega,\\ \nonumber
     			z_{k,M}=0&\mbox{in}\quad \Omega,\\
     			z_{k,m}=0\quad&\mbox{on}\quad \Gamma.
     		\end{cases}
     	\end{equation}
     	Multiplying $y_{k,m}$ on both sides of the above equation,  integration by parts in  $I_m\times\Omega$ yields
     	\begin{equation}
     		-\left([z_k]_m,y_{k,m}\right)+(\nabla z_k,\nabla y_k)_{L^2(I_m;L^2(\Omega))}-(\partial_n z_k,u_k)_{L^2(I_m;L^2(\Gamma))}=\Vert y_k\Vert^2_{L^2(I_m;L^2(\Omega))}.\nonumber
     	\end{equation}
     Summing up the above identity for $m$ and observing $z_{k,M}=z(T)=0$, we have
     	\begin{equation}
     		-\!\!\sum\limits_{m=1}^{M-1}([z_k]_m,y_{k,m})+(z^-_{k,M},y_{k,M})+(\nabla z_k,\nabla y_k)_{L^2(I;L^2(\Omega))}-(\partial_n z_k,u_k)_{L^2(I;L^2(\Gamma))}=\Vert y_k\Vert^2_I.\nonumber
     	\end{equation}
     	That is,
     	\begin{equation}
     		B(y_k,z_k)-(\partial_n z_k,u_k)_{L^2(I;L^2(\Gamma))}=\Vert y_k\Vert^2_I.\nonumber
     	\end{equation}

     	By using equation (\ref{EQF_1}), we derive
     	\begin{align*}
     		\Vert y_k\Vert^2_I&=-(\partial_n z_k,u_k)_{L^2(I;L^2(\Gamma))}\\
     		&\le \left\Vert \partial_n z_k\right\Vert_{L^2(I;L^2(\Gamma))}\left\Vert u_k\right\Vert_{L^2(I;L^2(\Gamma))}\\
     		&\le C\Vert z_k\Vert_{L^2(I;H^2(\Omega))}\Vert u_k\Vert_{L^2(I;L^2(\Gamma))}\\
     		&\le C(\Vert z_k\Vert_I+\Vert\Delta z_k\Vert_I)\Vert u_k\Vert_{L^2(I;L^2(\Gamma))}\\
     		&\le C\Vert y_k\Vert_I\Vert u_k\Vert_{L^2(I;L^2(\Gamma))},
     	\end{align*}
     where we have used the stability of the temporal semi-discrete solution of parabolic equations (cf. \cite[Corollary 4.2]{MeidnerVexler}). This completes the proof.
     \end{proof}
		
	Now we are ready to state the result on the error estimation for the solution of the temporal semi-discrete optimal control problem \eqref{time_semi_control}.
 \begin{Proposition}\label{Error_Tem_Opt_pair}
Assume that $\Omega$ is a  bounded, smooth domain or convex polytope. Let $(\bar{u},\bar{y})$ and  $(\bar{u}_k,\bar{y}_k)\in U_{ad}\times X_k$ be the optimal solution of the control problem \eqref{BC_Functioal} and the semi-discrete problem \eqref{time_semi_control}, respectively. Assume that $\bar{z}\in L^2(I;H^2(\Omega)\cap H^1_0(\Omega))$ is the adjoint state, $y_k(\bar{u})$ is the solution of \eqref{time_semi_state_e} with $u$ replaced by $\bar{u}$, and $\hat{z}_k\in X^0_k$ satisfies \eqref{EE10}. Then 
		\begin{equation}\label{EE00}
			\left\Vert \bar{u}_k-\bar{u}\right\Vert_{L^2(I;L^2(\Gamma))}+\left\Vert \bar{y}_k-\bar{y}\right\Vert_{L^2(I;L^2(\Omega))}\le C\left(\Vert \partial_n(\bar{z}-\hat{z}_k)\Vert_{L^2(I;L^2(\Gamma))}+\Vert \bar{y}-y_k(\bar{u})\Vert_{L^2(I;L^2(\Omega))}\right).
		\end{equation}
\end{Proposition}
\begin{proof}
Since $\hat{J}_k$ is a quadratic functional, $\hat{J}^{''}_k(u)$ is independent of $u$ for all $u\in U_{ad}$. Note that  $\hat{J}^{''}_k(u)(v,v)\ge\alpha\Vert v\Vert^2_{L^2(I;L^2(\Gamma))}$, $\forall v\in L^2(I;L^2(\Gamma))$. Define the auxiliary variable $\tilde{z}_k\in X^0_k $ such that
		\begin{equation}
			B(\varphi_k,\tilde{z}_k)=(y_k(\bar{u})-y_d,\varphi_k)_I\qquad\forall\varphi_k\in X^0_k.\nonumber
		\end{equation}
 Then
		\begin{align}\label{EE0}\notag
			\alpha\Vert \bar{u}-\bar{u}_k\Vert^2_{L^2(I;L^2(\Gamma))}&\le \hat{J}^{''}_k(\bar{u}_k)(\bar{u}-\bar{u}_k,\bar{u}-\bar{u}_k)\\\notag
			&= \hat{J}^{'}_k(\bar{u})(\bar{u}-\bar{u}_k)-\hat{J}^{'}_k(\bar{u}_k)(\bar{u}-\bar{u}_k)\\\notag
			&\le \hat{J}^{'}_k(\bar{u})(\bar{u}-\bar{u}_k)-\hat{J}^{'}(\bar{u})(\bar{u}-\bar{u}_k)\\\notag
			&\le \int_{\Sigma_T}(\partial_n\bar{z}-\partial_n\tilde{z}_k)\big(\bar{u}-\bar{u}_k\big)\\
			&\le \left\Vert \partial_n\bar{z}-\partial_n\tilde{z}_k\right\Vert_{L^2(I;L^2(\Gamma))}\Vert \bar{u}-\bar{u}_k\Vert_{L^2(I;L^2(\Gamma))},
		\end{align}
		where we have used  $-\hat{J}^{'}_k(\bar{u}_k)(\bar{u}-\bar{u}_k)\le0\le-\hat{J}^{'}(\bar{u})(\bar{u}-\bar{u}_k)$ which can be easily checked by taking $v=\bar{u}$ in (\ref{semi_discrete_VI_1}) and $v=\bar{u}_k$ in (\ref{V_Inequation2}). From the inequality (\ref{EE0}), we can obtain
		\begin{equation}\label{EE1}
			\begin{split}
				\alpha\left\Vert \bar{u}-\bar{u}_k\right\Vert_{L^2(I;L^2(\Gamma))}&\le\Vert \partial_n\bar{z}-\partial_n\tilde{z}_k\Vert_{L^2(I;L^2(\Gamma))}\\
				&\le \Vert \partial_n\bar{z}-\partial_n\hat{z}_k\Vert_{L^2(I;L^2(\Gamma))}+\Vert\partial_n\hat{z}_k-\partial_n\tilde{z}_k\Vert_{L^2(I;L^2(\Gamma))},
			\end{split}
		\end{equation}
		where $\hat{z}_k\in X^0_k$ satisfies \eqref{EE10}.
		It follows from the stability estimate of the temporal semi-discrete solution of adjoint equations (cf. \cite[Corollary 4.2]{MeidnerVexler}) that
		\begin{equation}\label{EE2}
			\begin{aligned}
				\Vert\partial_n\hat{z}_k-\partial_n\tilde{z}_k\Vert_{L^2(I;L^2(\Gamma))}&\le C\Vert\hat{z}_k -\tilde{z}_k\Vert_{L^2(I;H^2(\Omega))}\\
				&\le C\big(\Vert\hat{z}_k -\tilde{z}_k\Vert_{L^2(I;L^2(\Omega))}+\Vert\Delta(\hat{z}_k -\tilde{z}_k)\Vert_{L^2(I;L^2(\Omega))}\big)\\
				&\le C\Vert\bar{y}-y_k(\bar{u})\Vert_I.
			\end{aligned}
		\end{equation}


Now, we consider the error estimate for the semi-discrete optimal state:
		\begin{equation}\label{EE3}
			\begin{split}
				\Vert \bar{y}-\bar{y}_k\Vert_I&\le \Vert \bar{y}-y_k(\bar{u})\Vert_I+\Vert y_k(\bar{u})-\bar{y}_k\Vert_I\\
				&\le \Vert \bar{y}-y_k(\bar{u})\Vert_I+C\Vert\tilde{P}_k\bar{u}-\tilde{P}_k\bar{u}_k\Vert_{L^2(I;L^2(\Gamma))}\\
				&\le \Vert \bar{y}-y_k(\bar{u})\Vert_I+C\Vert\bar{u}-\bar{u}_k\Vert_{L^2(I;L^2(\Gamma))},\\
			\end{split}
		\end{equation}
		where we have used Lemma \ref{semi_stable_estimate} and the stability of $L^2$-projection $\tilde{P}_k$. Combining the estimates \eqref{EE1}-\eqref{EE3}, we finally obtain the estimate \eqref{EE00}. This completes the proof.
\end{proof}

From Proposition \ref{Error_Tem_Opt_pair}, it remains to estimate the right-hand terms of \eqref{EE00}, which  depends heavily on the smoothness of $\Omega$.
	\begin{Theorem}\label{time_semi_control_error}
Assume that  $f\in L^2(I;L^2(\Omega)),\ y_0\in H^{2s-1}(\Omega)$ and $y_d\in L^2(I;H^{2s^\prime-\frac{1}{2}}(\Omega))\cap H^{s^\prime-\frac{1}{4}}(I;L^2(\Omega))$, where $s\in[\frac{1}{2}, \frac{3}{4})$, $s^\prime\in [s-\frac{1}{4}, s]$ in case that $\Omega$ is smooth, and $s=\frac{1}{2},\ s^\prime=\frac{1}{4}$ in case that $\Omega$ is a polytope. Let $(\bar{u},\bar{y})$ and $(\bar{u}_k,\bar{y}_k)\in U_{ad}\times X_k$ be the optimal solution of the optimal control problem \eqref{BC_Functioal} and the semi-discrete optimal control problem \eqref{time_semi_control}, respectively. Then 
		\begin{equation}
			\Vert \bar{u}_k-\bar{u}\Vert_{L^2(I;L^2(\Gamma))}+\Vert \bar{y}_k-\bar{y}\Vert_{L^2(I;L^2(\Omega))}\le C(k^s+k^{s^\prime}).
		\end{equation}
	\end{Theorem}
	\begin{proof}
By Proposition \ref{Error_Tem_Opt_pair}, we only need to consider the two terms on the right-hand side of \eqref{EE00}. In viewing of the regularity result of the optimal pair $(\bar{u},\bar{y})$ in Theorems \ref{optimall_regularity} and \ref{regularity_cont_cnvex}, we first use Theorem \ref{error_semi_d} to deduce the following estimate:
\begin{equation}\label{Estimate_SE}
			\Vert \bar{y}-{y}_k(\bar{u})\Vert_{L^2(I;L^2(\Omega))}\le Ck^s,
		\end{equation}
where $s=\frac{1}{2}$ for a polytopal domain $\Omega$ and $\frac{1}{2}\le s<\frac{3}{4}$ for a smooth domain.


Based on the regularity of the adjoint state $\bar{z}$ in Theorems \ref{optimall_regularity} and \ref{regularity_cont_cnvex}, we estimate the first term on the right-hand side of \eqref{EE00} by taking $s^\prime=\frac{1}{4}$ for a polytopal $\Omega$, and $s-\frac{1}{4}\le s^\prime\le s$ for a smooth domain. By the trace theorem, we have
	\begin{equation}\label{WE1}
\begin{split}
			\Vert\partial_n\bar{z}-\partial_n\hat{z}_k\Vert_{L^2(I;L^2(\Gamma))}&\le C\Vert\bar{z}-\hat{z}_k\Vert^{\frac{1}{2}}_{L^2(I;H^2(\Omega))}\Vert\bar{z}-\hat{z}_k\Vert^{\frac{1}{2}}_{L^2(I;H^1(\Omega))}\\
			&\le C\Vert\bar{z}-\hat{z}_k\Vert^{1-s^\prime}_{L^2(I;H^{2s^\prime+\frac{3}{2}}(\Omega))}\Vert\bar{z}-\hat{z}_k\Vert^{s^\prime}_{L^2(I;H^{2(s^\prime-\frac{1}{4})}(\Omega))}\\
		&\le
Ck^{s^\prime},
\end{split}
\end{equation}
where we have used Lemma \ref{Eer_S} and the following interpolation result (see, e.g., \cite{J.L1}):
\begin{align*}
&\left[L^2(I;H^{2s^\prime+\frac{3}{2}}(\Omega)), L^2(I;H^{2(s^\prime-\frac{1}{4})}(\Omega))\right]_{s^\prime-\frac{1}{4}}=L^2(I;H^2(\Omega)),\\
&\left[L^2(I;H^{2s^\prime+\frac{3}{2}}(\Omega)), L^2(I;H^{2(s^\prime-\frac{1}{4})}(\Omega))\right]_{s^\prime+\frac{1}{4}}=L^2(I;H^1(\Omega)).
\end{align*}
Combining the estimates \eqref{Estimate_SE} and \eqref{WE1}, we finish the proof.	
\end{proof}

\subsection{Error estimates for the spatial discretization }
	In this subsection, we estimate the error between the fully discrete solution of optimal control problem \eqref{space_semi_control} and the temporal semi-discrete solution of optimal control problem \eqref{time_semi_control}. To this end, we first establish the stability of the fully discrete solution to the parabolic  equation with respect to inhomogeneous Dirichlet data.
	
	 \begin{Proposition}\label{full_stable}
	 	For any given $g\in L^2(I;L^2(\Omega))$, let $z\in L^2(I;H^2(\Omega)\cap H_0^1(\Omega))\cap H^1(I;L^2(\Omega))$ be the solution of equation \eqref{T_back_eq} and $z_k$ be the corresponding temporal semi-discretization. By the identity \eqref{semi_discrete_norm_deriv},   the normal derivative $\partial_nz_k$ satisfies
	 	\begin{eqnarray}
	 		\int_{\Sigma_T}\partial_nz_k\phi_kdsdt=-\int_{\Omega_T}gp_k(\phi_k)dxdt\qquad \forall \phi_k\in X_k(\Gamma),\label{norm_dev_time_semi}
	 	\end{eqnarray}
	 	where $p_k(\phi_k)\in X_k$ is the solution of \eqref{semi_discrete_VI_2}. Denote  by $z_{kh}$ the fully discrete solution of \eqref{T_back_eq} and by $\partial^h_nz_{kh}$  the fully discrete normal derivative of $z_{kh}$. Let $y_{kh}\in X_{kh}$ be the solution of
	 	\begin{equation}\nonumber
	 			B(y_{kh},\varphi_{kh})=0\qquad\forall\varphi_{kh}\in X^0_{kh},\quad
	 			y_{kh}|_{I\times\Gamma}=u_{kh}\in X_{kh}(\Gamma).
	 	\end{equation}
Then there holds
	 	\begin{equation}\label{result_1}
	 		\begin{split}
	 			\left\Vert \partial^h_nz_{kh}\right\Vert_{L^2(I;L^2(\Gamma))}&\le C\Vert g\Vert_I,\\
	 			\Vert y_{kh}\Vert_I&\le C\Vert u_{kh}\Vert_{L^2(I;L^2(\Gamma))},\\
	 			\Vert \partial_nz_k-\partial^h_nz_{kh}\Vert_{L^2(I;L^2(\Gamma))}&\le Ch^{\frac{1}{2}}\Vert g\Vert_I.
	 		\end{split}
	 	\end{equation}
	 \end{Proposition}
	 \begin{proof}
	 	We first prove the third estimate of \eqref{result_1}. For any given $\omega_{kh}\in X_{kh}(\Gamma)$, setting $\phi_k=\omega_{kh}$ in \eqref{norm_dev_time_semi} and $\phi_{kh}=\omega_{kh}$ in \eqref{full_discrete_norm_div}, then there holds
	 	\begin{equation}
	 		\begin{split}
	 			\int_{\Sigma_T}\partial_nz_k\omega_{kh}dsdt&=-\int_{\Omega_T}gp_k(\omega_{kh})dxdt\\
	 			&=-\int_{\Omega_T}gp_{kh}(\omega_{kh})dxdt+\int_{\Omega_T}g(p_{kh}(\omega_{kh})-p_k(\omega_{kh}))dxdt\\
	 			&=\int_{\Sigma_T}\partial^h_nz_{kh}\omega_{kh}dsdt+\int_{\Omega_T}g(p_{kh}(\omega_{kh})-p_k(\omega_{kh}))dxdt.
	 		\end{split}
	 	\end{equation}
	 	
	 	Define the Ritz projection $R_{h}:X_{kh}\to X^0_{kh}$ such that $R^n_{h}p_{kh,n}:=R_{h}p_{kh}|_{I_n}\in V^0_{h}$ satisfies
	 	\begin{equation}
	 		(\nabla R^n_{h}p_{kh,n},\nabla v_h)=(\nabla p_{kh,n},\nabla v_h)\qquad\forall v_h\in V^0_h,\quad \quad n=1,\cdots,M.
	 	\end{equation}
Using the fact  $p_{kh}(\omega_{kh})-p_k(\omega_{kh})\in X^0_k$ and $B(p_k(\omega_{kh}),z_k)=B(p_{kh}(\omega_{kh}),z_{kh})=B(R_{h}p_{kh}(\omega_{kh}),z_k-z_{kh})=0$, we have
	 	\begin{eqnarray}\notag
	 		&&\Big|\int_{\Sigma_T}(\partial_nz_k-\partial^h_nz_{kh})\omega_{kh}dsdt\Big|
	 		=\Big|\int_{\Omega_T}g(p_{kh}(\omega_{kh})-p_k(\omega_{kh}))dxdt\Big|\\\notag
	 		&=&\left\vert B(p_{kh}(\omega_{kh})-p_k(\omega_{kh}),z_k)\right\vert\\\notag
	 		&=&\vert B(p_{kh}(\omega_{kh})-R_{h}p_{kh}(\omega_{kh}),z_k-z_{kh})\vert\\\notag
	 		&=&\Big\vert (\nabla (p_{kh}(\omega_{kh})-R_{h}p_{kh}(\omega_{kh})),\nabla (z_k-z_{kh}))_I\\\notag
			&&-\sum\limits_{m=1}^M\Big((p_{kh}(\omega_{kh})-R_{h}p_{kh}(\omega_{kh}))^-_m,[z_k-z_{kh}]_m\Big)\Big|\\\notag
	 		&\le& \left\Vert \nabla (p_{kh}(\omega_{kh})-R_{h}p_{kh}(\omega_{kh}))\right\Vert_I \Vert\nabla (z_k-z_{kh})\Vert_I\\\notag
			&&+\Big(\sum\limits_{m=1}^Mk_m\Vert p_{kh,m}(\omega_{kh})-R^m_{h}p_{kh,m}(\omega_{kh})\Vert^2_{L^2(\Omega)}\Big)^{\frac{1}{2}}\Big(\sum\limits_{m=1}^M\frac{\Vert[z_k-z_{kh}]_m\Vert^2_{L^2(\Omega)}}{k_m}\Big)^{\frac{1}{2}}\\\notag
	 		&=& \Vert\nabla (p_{kh}(\omega_{kh})-R_{h}p_{kh}(\omega_{kh}))\Vert_I \Vert\nabla (z_k-z_{kh})\Vert_I\\\notag
			&&+\Vert p_{kh}(\omega_{kh})-R_{h}p_{kh}(\omega_{kh})\Vert_I\Big(\sum\limits_{m=1}^M\frac{\Vert[z_k-z_{kh}]_m\Vert^2_{L^2(\Omega)}}{k_m}\Big)^{\frac{1}{2}}\\
	 		&\le& C h^{\frac{1}{2}}\Vert g\Vert_I\Vert\omega_{kh}\Vert_{L^2(I;L^2(\Gamma))},\label{estimate_norm_dev}
	 	\end{eqnarray}
	 	where we have used the following estimates (cf. \cite{WM-Hinze,MeidnerVexler}):
	 	\begin{align*}
	 		\Big(\sum\limits_{m=1}^M\frac{\left\Vert[z_k-z_{kh}]_m\right\Vert^2_{L^2(\Omega)}}{k_m}\Big)^{\frac{1}{2}}&\le C\Vert g\Vert_I,\\
	 		\Vert \nabla (z_k-z_{kh})\Vert_I&\le Ch\Vert g\Vert_I,\\
	 		\Vert p_{kh}(\omega_{kh})-R_{h}p_{kh}(\omega_{kh})\Vert_I&\le Ch^{\frac{1}{2}}\Vert\omega_{kh}\Vert_{L^2(I;L^2(\Gamma))},\\
	 		\Vert \nabla (p_{kh}(\omega_{kh})-R_{h}p_{kh}(\omega_{kh}))\Vert_I&\le Ch^{-\frac{1}{2}}\Vert\omega_{kh}\Vert_{L^2(I;L^2(\Gamma))}.
	 	\end{align*}
	 	
	 	Particularly, setting $\omega_{kh}=\tilde{P}_{kh}\partial_nz_k-\partial_n^hz_{kh}$ in \eqref{estimate_norm_dev}, we obtain
	 	\begin{align*}
	 		\left\Vert\tilde{P}_{kh}\partial_nz_k-\partial_n^hz_{kh}\right\Vert_{L^2(I;L^2(\Gamma))}^2&=\Big|\int_{\Sigma_T}(\partial_nz_k-\partial^h_nz_{kh})\omega_{kh}dsdt\Big|\\
	 		&\le Ch^{\frac{1}{2}}\Vert\tilde{P}_{kh}\partial_nz_k-\partial_n^hz_{kh}\Vert_{L^2(I;L^2(\Gamma))}\Vert g\Vert_I.
	 	\end{align*}
	 	Using the triangle inequality yields
	 	\begin{align*}
	 		\Vert\partial_nz_k-\partial_n^hz_{kh}\Vert_{L^2(I;L^2(\Gamma))}&\le \Vert\partial_nz_k-\tilde{P}_h\partial_nz_k \Vert_{L^2(I;L^2(\Gamma))}+\Vert\tilde{P}_h\partial_nz_k-\partial_n^hz_{kh}\Vert_{L^2(I;L^2(\Gamma))}\\
	 		&\le Ch^{\frac{1}{2}}\Vert\partial_nz_k\Vert_{L^2(I;H^{\frac{1}{2}}(\Gamma))}+C h^{\frac{1}{2}}\Vert g\Vert_I\\
	 		&\le C h^{\frac{1}{2}}\Vert g\Vert_I.
	 	\end{align*}
	 This verifies the third estimate in \eqref{result_1}.
	
	 Next, we prove the first estimate in \eqref{result_1}. It follows from the triangle inequality that
	 	\begin{align*}
	 		\Vert\partial_n^hz_{kh}\Vert_{L^2(I;L^2(\Gamma))}\le \Vert\partial_nz_k\Vert_{L^2(I;L^2(\Gamma))}+\Vert\partial_nz_k-\partial_n^hz_{kh}\Vert_{L^2(I;L^2(\Gamma))}
	 		\le C(\Vert g\Vert_I+ h^{\frac{1}{2}}\Vert g\Vert_I)
	 		\le C \Vert g\Vert_I.
	 	\end{align*}	
	 	
	 	Finally, we verify the second estimate in \eqref{result_1}. Taking $g=y_{kh}$ on the right-hand side of \eqref{T_back_eq} and denoting the fully discrete solution of \eqref{T_back_eq} by $z_{kh}$, then the solution $p_{kh}(\phi_{kh})$ of \eqref{Test_EQ}  satisfies $p_{kh}(\phi_{kh})=y_{kh}$ by taking $\phi_{kh}=u_{kh}$ in \eqref{full_discrete_norm_div}. Therefore, there holds
	 	\begin{align*}
	 		\int_{\Omega_T}y_{kh}^2dxdt=-\int_{\Sigma_T}\partial_n^hz_{kh}u_{kh}dsdt
	 		\le \left\Vert \partial_n^hz_{kh}\right\Vert_{L^2(I;L^2(\Gamma))}\Vert u_{kh}\Vert_{L^2(I;L^2(\Gamma))}
	 		\le C\Vert y_{kh}\Vert_I\Vert u_{kh}\Vert_{L^2(I;L^2(\Gamma))},
	 	\end{align*}
	 	where we have used the first and the third estimates in \eqref{result_1}.  This completes the proof.
	 \end{proof}
	
	 \begin{Theorem}\label{FD_error_estimate}
	 	Let $(\bar{u}_k,\bar{y}_k)\in U_{ad}\times X_k$ and $(\bar{u}_{kh},\bar{y}_{kh})\in U_{ad}\times X_{kh}$ be the optimal pair of the temporal semi-discrete control problem \eqref{time_semi_control} and  the fully discrete optimal control problem \eqref{space_semi_control}, respectively.  Then 
	 	\begin{equation}
	 		\Vert\bar{u}_k-\bar{u}_{kh}\Vert_{L^2(I;L^2(\Gamma))}+\Vert \bar{y}_k-\bar{y}_{kh}\Vert_{L^2(I;L^2(\Omega))}\le Ch^{\frac{1}{2}}.\nonumber
	 	\end{equation}
	 \end{Theorem}
	 \begin{proof}
	 	Since $\hat{J}_{kh}(u)$ is a quadratic functional with respect to $u$, $\hat{J}^{''}_{kh}(u)$ is a constant operator and satisfies $\hat{J}^{''}_{kh}(u)\ge\alpha$. Taking $v=\bar{u}_k$ and $v=\bar{u}_{kh}$ in \eqref{optimal_condition_adjoint} and \eqref{semi_discrete_VI_1}, respectively, there holds
	 	$$-\hat{J}^{'}_{kh}(\bar{u}_{kh})(\bar{u}_k-\bar{u}_{kh})\le0\le-\hat{J}^{'}_k(\bar{u}_k)(\bar{u}_k-\bar{u}_{kh}).$$
	 	Hence
	 	\begin{equation}\label{EQQ}
	 	\begin{split}
	 		\alpha\Vert \bar{u}_k-\bar{u}_{kh}\Vert^2_{L^2(I;L^2(\Gamma))}&\le \hat{J}^{''}_{kh}(\bar{u}_{kh})(\bar{u}_k-\bar{u}_{kh},\bar{u}_k-\bar{u}_{kh})\\
	 		&= \hat{J}^{'}_{kh}(\bar{u}_k)(\bar{u}_k-\bar{u}_{kh})-\hat{J}^{'}_{kh}(\bar{u}_{kh})(\bar{u}_k-\bar{u}_{kh})\\
	 		&\le \hat{J}^{'}_{kh}(\bar{u}_k)(\bar{u}_k-\bar{u}_{kh})-\hat{J}^{'}_k(\bar{u}_k)(\bar{u}_k-\bar{u}_{kh})\\
	 		&\le \int_{\Sigma_T}(\partial_n\bar{z}_k-\partial^h_n z_{kh})(\bar{u}_k-\bar{u}_{kh})dsdt\\
	 		&\le \Vert \partial_n\bar{z}_k-\partial^h_nz_{kh}\Vert_{L^2(I;L^2(\Gamma))}\Vert \bar{u}_k-\bar{u}_{kh}\Vert_{L^2(I;L^2(\Gamma))},
	 	\end{split}
	 	\end{equation}
	 	where $z_{kh}\in X^0_{kh}$ satisfies
	 	\begin{equation}
	 		B(\varphi_{kh},z_{kh})=(\varphi_{kh},y_{kh}(\bar{u}_k)-y_d)_I\qquad\forall\varphi_{kh}\in X^0_{kh}.\nonumber
	 	\end{equation}
	 	Let $\hat{z}_{kh}\in X^0_{kh}$ be the solution of the equation
	 	\begin{equation}
	 		B(\varphi_{kh},\hat{z}_{kh})=(\varphi_{kh},\bar{y}_k-y_d)_I\qquad\forall\varphi_{kh}\in X^0_{kh}.\nonumber
	 	\end{equation}
	 	By using the estimate \eqref{EQQ} and the triangle inequality, there holds
	 	\begin{equation}
	 		\begin{split}
	 			\alpha\Vert \bar{u}_k-\bar{u}_{kh}\Vert_{L^2(I;L^2(\Gamma))}&\le \left\Vert\partial_n\bar{z}_k-\partial^h_nz_{kh}\right\Vert_{L^2(I;L^2(\Gamma))}\\
	 			&\le \left\Vert\partial_n\bar{z}_k-\partial^h_n\hat{z}_{kh}\right\Vert_{L^2(I;L^2(\Gamma))}+\left\Vert\partial^h_n\hat{z}_{kh}-\partial^h_nz_{kh}\right\Vert_{L^2(I;L^2(\Gamma))}\\
	 			&\le Ch^{\frac{1}{2}}\left\Vert\bar{y}_k-y_d\right\Vert_I+C\left\Vert \bar{y}_k-y_{kh}(\bar{u}_k)\right\Vert_I\\\nonumber
	 			&\le Ch^{\frac{1}{2}}\left\Vert\bar{y}_k-y_d\right\Vert_I+Ch\big(\Vert \bar{y}_k\Vert_{L^2(I;H^1(\Omega))}+\Vert \bar u\Vert_{L^2(I;H^{\frac{1}{2}}(\Gamma))}\big)\\
	 			&\le Ch^{\frac{1}{2}},
	 		\end{split}
	 	\end{equation}
	 	where we have used Lemma \ref{back_regularity}, Theorem \ref{error_space_esti}, Propositions \ref{full_stable} and \ref{semi_regularity}.
	 	
	 	Finally, the error between the semi-discrete  and the fully discrete states is given by
	 	\begin{align*}
	 		\Vert\bar{y}_k-\bar{y}_{kh}\Vert_I&\le \Vert\bar{y}_k-y_{kh}(\bar{u}_k)\Vert_I+\Vert y_{kh}(\bar{u}_k)-\bar{y}_{kh}\Vert_I\\
	 		&\le Ch+C\Vert\tilde{P}_{kh}(\bar{u}_k-\bar{u}_{kh})\Vert_I\\
	 		&\le Ch+C\Vert\bar{u}_k-\bar{u}_{kh}\Vert_I\\
	 		&\le Ch^{\frac{1}{2}},
	 	\end{align*}
	where we used Proposition \ref{full_stable}. This completes the proof.
	 \end{proof}
	
	 By combining Theorems \ref{time_semi_control_error} and \ref{FD_error_estimate}, we can obtain the main result on the error estimation between the optimal pairs of the continuous and the fully discrete control problems.
	 \begin{Theorem}
	 	Let $(\bar{u},\bar{y})\in U_{ad}\times L^2(I;H^1(\Omega))\cap H^{\frac{1}{2}}(I;L^2(\Omega))$ and $(\bar{u}_{kh},\bar{y}_{kh})\in U_{ad}\times X_{kh}$ be the solution of the continuous control problem \eqref{BC_Functioal} and   the fully discrete control problem \eqref{FD_error_estimate}, respectively. Then we have 
	 	\begin{equation}\label{main_result}
	 		\Vert\bar{u}-\bar{u}_{kh}\Vert_{L^2(I;L^2(\Gamma))}+\Vert \bar{y}-\bar{y}_{kh}\Vert_{L^2(I;L^2(\Omega))}\le C(h^{\frac{1}{2}}+k^{\frac{1}{4}}).
	 	\end{equation}
	 \end{Theorem}
	
   \begin{Remark}  There is another error estimate for $(\bar{u}_{kh},\bar{y}_{kh})$ in \cite{WM-Hinze} in the  two dimensional case, where the condition $k=O(h^2)$ is imposed to ensure a convergence order similar to \eqref{main_result}. More precisely, the estimate for $(\bar{u}_{kh},\bar{y}_{kh})$ in \cite{WM-Hinze} reads
	 \begin{equation}\label{other}
	 	\Vert\bar{u}-\bar{u}_{kh}\Vert_{L^2(I;L^2(\Gamma))}+\Vert \bar{y}-\bar{y}_{kh}\Vert_{L^2(I;L^2(\Omega))}\le C(h^{\frac{1}{2}}+k^{\frac{1}{4}}+h^{\frac{3}{2}}k^{-\frac{1}{2}}+h^{-\frac{1}{2}}k^{\frac{1}{2}}+h^{\frac{5}{2}}k^{-1}).
	 \end{equation}
Here the time step $k$ and mesh size $h$ are coupled. Therefore, the condition $k=O(h^2)$ is necessary to  obtain the optimal convergence order. However, the condition $k=O(h^2)$ is unnecessary both in theory and numerical computations. In this paper we remove this mesh size condition by developing new a priori error analysis, and thus improve the error estimate in \cite{WM-Hinze}.
\end{Remark}
\begin{Remark}
  For the fully discretization of parabolic Dirichlet boundary control problems in smooth domains, the traditional polygonal approximation will introduce geometric errors which make the error estimate more complicate. We refer to \cite{DeckelnickGuntherHinze2009} for the elliptic case and postpone the error analysis to a future work.
 The extensive numerical results for parabolic Dirichlet boundary control problems can be found in \cite{WM-Hinze} and are omitted here.
\end{Remark}



\medskip
\end{document}